\documentclass{elsarticle}

\usepackage{lineno,hyperref}
\usepackage[margin=1in]{geometry}

\usepackage{amssymb}
\usepackage{amsmath}
\usepackage{color}
\usepackage{tikz}
\usetikzlibrary{arrows,automata,backgrounds,calendar,mindmap,trees}
\usepackage{xcolor}

\newcommand{\uu}{\underline{u}}
\newcommand{\uv}{\underline{v}}
\newcommand{\ou}{\overline{u}}
\newcommand{\ov}{\overline{v}}

\newcommand{\iTO}{\int_0^T\!\!\int_{\Omega}}
\newcommand{\iTG}{\int_0^T\!\int_{\Gamma}}

\numberwithin{equation}{section}
\newtheorem{theorem}{Theorem}[section]
\newtheorem{lemma}[theorem]{Lemma}
\newtheorem{proposition}[theorem]{Proposition}

\newtheorem{definition}{Definition}[section]
\newtheorem{remark}{Remark}[section]

\newproof{proof}{Proof}

\begin{document}
	
	\begin{frontmatter}
		
		\title{Well-posedness and exponential equilibration of a volume-surface reaction-diffusion system with nonlinear boundary coupling}
		
		
		\author[mymainaddress]{Klemens Fellner\corref{mycorrespondingauthor}}
		\cortext[mycorrespondingauthor]{Corresponding author}
		\ead{klemens.fellner@uni-graz.at}
		
		\author[mymainaddress]{Evangelos Latos}
		\ead{evangelos.latos@uni-graz.at}
		
		\author[mymainaddress]{Bao Quoc Tang}
		\address[mymainaddress]{Institute of Mathematics and Scientific Computing, University of Graz, Heinrichstra{\ss}e 36, 8010 Graz, Austria}
		\ead{quoc.tang@uni-graz.at}		
		
		\begin{abstract}
			We consider a model system consisting of two reaction-diffusion equations, 
			where one species diffuses in a volume while the
			other species diffuses on the surface which surrounds the volume. The two equations are coupled via a nonlinear reversible Robin-type boundary condition for the volume species and a matching reversible source term for the boundary species. As a consequence of the coupling, the total mass of the two species is conserved. 
			The considered system is motivated for instance by models for asymmetric stem cell division.
			
			Firstly we prove the existence of a unique weak solution via an iterative method of converging upper and lower solutions to overcome the difficulties of the nonlinear boundary terms. Secondly, our main result shows explicit exponential convergence to equilibrium via an entropy method after deriving a suitable entropy entropy-dissipation estimate for the considered nonlinear volume-surface reaction-diffusion system.
		\end{abstract}
		
		
	\end{frontmatter}

	
	\section{Introduction}
	In this paper, we consider a nonlinear volume-surface reaction-diffusion system, which couples a non-negative volume-concentration $u(x,t)$ diffusing on a bounded domain $\Omega \subset \mathbb R^N (N\geq 1)$ with a non-negative surface-concentration $v(x,t)$ diffusing on the sufficiently smooth boundary $\Gamma:= \partial\Omega$ of $\Omega$ (e.g. $\partial\Omega\in C^{2+\epsilon}$ for $\epsilon>0$). 
	
	The interface conditions connecting these two concentrations are a nonlinear Robin-type boundary condition for the volume-concentration $u(x,t)$ and a matching reversible reaction source term in the equation for the surface-concentration $v(x,t)$: 
	\begin{equation}
	\begin{cases}
	u_t - \delta_{u}\Delta u = 0, &x\in\Omega, t>0,\\
	\delta_u\frac{\partial u}{\partial \nu} = -\alpha(k_u u^{\alpha} - k_v v^{\beta}),&x\in\Gamma, t>0,\\
	v_t - \delta_{v}\Delta_{\Gamma}v= \beta(k_u u^{\alpha} - k_v v^{\beta}), &x\in\Gamma, t>0,\\
	u(0,x) = u_0(x)\ge0, &x\in\Omega,\\
	v(0,x) = v_0(x)\ge0, &x\in\Gamma.
	\end{cases}
	\label{e1}
	\end{equation} 
	Here, we denote by $\Delta$ the Laplace operator on $\Omega$ with a positive diffusion coefficient $\delta_u>0$ and by $\Delta_{\Gamma}$ the Laplace-Beltrami operator on $\Gamma$ (see e.g. \cite{GT}) with a non-negative diffusion coefficient $\delta_v \geq 0$, and $\nu(x)$ denotes the unit outward normal vector of $\Gamma$ at the point $x$. Moreover, we shall consider nonnegative initial concentrations $u_0(x)\geq 0$ on $\Omega$ and $v_0(x)\geq 0$ on $\Gamma$. 
	
	{The stoichiometric coefficients $\alpha, \beta\in [1,\infty)$ together with the positive, bounded reaction rates $k_u(t,x)$, $k_v(t,x) \in L_{+}^{\infty}([0,\infty)\times \Gamma)$ characterise the key feature of the model system \eqref{e1}, which is the nonlinear reversible reaction between the volume density $u(t,x)$ and 
		the surface density $v(t,x)$ located at the boundary $\Gamma$.}
	
	We emphasise that the reversible reaction between volume- and boundary-concentrations in system \eqref{e1} {\it preserves the total initial mass} $M$, which shall be assumed positive in the following:
	\begin{align}
	\label{cons}
	M=\beta\int_{\Omega}u(t,x)\,dx + \alpha\int_{\Gamma}v(t,x)\,dS
	=\beta\int_{\Omega}u_0(x)\,dx + \alpha\int_{\Gamma}v_0(x)\,dS>0,\qquad \forall t\ge 0.
	\end{align}
	
	The study of system \eqref{e1} is motivated by models of \emph{asymmetric stem cell division}. In stem cells undergoing asymmetric cell division, particular proteins (so-called cell-fate determinants) are localised in only one of the two daughter cells during mitosis. These cell-fate determinants trigger in the following the differentiation of one daughter cell into specific tissue while the other daughter cell remains a stem cell. 
	
	In Drosophila, SOP stem cells provide a well-studied biological example model of asymmetric stem cell division, 
	see e.g. \cite{BMK,MEBWK,WNK} and the references therein. The mechanism of asymmetric cell division in SOP stem cells operates around a key protein called Lgl (Lethal giant larvae), which exists in two conformational states: a non-phosphorylated form which regulates the localisation of the cell-fate-determinants in the membrane of one daughter cell, and a phosphorylated form which is inactive. 
	
	First mathematical models describing the evolution and localisation of phosphorylated and non-phospho\-ryl\-ated Lgl in SOP stem cells were presented in \cite{BFR, Ros} under the assumption of  linear phosphorylation and de-phosphorylation kinetics. However, it is known that Lgl offers three phosphorylation sites \cite{BMK}. Thus, if more than one site needs to be phosphorylated in order to effectively deactivate Lgl, a realistic model should rather consider nonlinear kinetics. 
	
	\medskip
	
	The system \eqref{e1} formulates a nonlinear mathematical core model, which strongly simplifies the biological model 
	for SOP stem cells by focussing only on the concentration $u(x,t)$ of the phosphorylated Lgl in the cytoplasm (i.e. in the cell volume) and the concentration $v(x,t)$
	of non-phosphorylated Lgl at the cortex/membrane of the cell. 
	The exchange of phosphorylated Lgl $u(x,t)$ and non-phosphorylated Lgl $v(x,t)$ is described by the above nonlinear reaction located at the boundary. The considered evolution process conserves the total mass of Lgl as quantified in the conservation law \eqref{cons}.
	
	Volume-surface reaction-diffusion systems describing models related to \eqref{e1} have recently gained rapidly increasing attention 
	as they occur naturally in many areas of applied mathematics as cell-biology, ecology and also fluid-dynamics, see e.g. \cite{AJLRRW,ElRa,FrNeRa,KD,MS,NGCRSS, BRR13a, BRR13b, BCRR15,  KD01, MCV15} and references therein. 
	
	\medskip
	
	The first aim of this paper is to prove the global existence of a unique weak solution to the model system \eqref{e1} {under certain technical assumptions on the reaction rates $k_u(t,x)$ and $k_v(t,x)$} (see Theorem \ref{theo:ExistenceAndUniqueness} below). The main difficulties arise from the arbitrary power-law nonlinearities located at the boundary $\Gamma$ and shall be overcomed by applying an iteration method of converging upper and lower solutions, in the spirit of e.g. \cite{Pao}. This method is based on proving a comparison principle for upper and lower solutions (see e.g. \cite{Souplet}), which so far - up to our knowledge - has not been established for volume-surface reaction-diffusion systems. Once the comparison principle is shown, the existence of weak solutions to \eqref{e1} follows from an iteration argument, which uses {the fact} that the involved nonlinearities are quasi-monotone non-decreasing. The existence of solutions to related linear models was proven in \cite{ElRa,FrNeRa} by fix-point methods. Our approach has the advantage of providing intrinsic a-priori bounds, which allows {us} to obtain global solutions {to} the superlinear problem \eqref{e1}.   
	\medskip
	
	 {In the} second part of the manuscript, our main result proves an  \emph{explicit exponential convergence to equilibrium} for the system \eqref{e1} via the so-called \emph{entropy method}. 
	The basic idea of the entropy method consists of studying the large-time asymptotics of a dissipative PDE model by looking for a nonnegative Lyapunov 
	functional $E(f)$ and its nonnegative dissipation 
	$$
	D(f)=-\frac{d}{dt} E(f(t))
	$$ 
	along the flow of the PDE model. We shall show that the entropy structure of system \eqref{e1} is
	well-behaved in the following sense: firstly, all states with $D(f)=0$, which also satisfy all the involved conservation laws, identify a unique entropy-minimising equilibrium $f_{\infty}$, i.e.  
	$$
	D(f) = 0\quad \text{and \quad conservation laws} \iff f=f_{\infty},
	$$ 
	and secondly, there exists an \emph{entropy entropy-dissipation 
		estimate} of the form 
	$$
	D(f) \ge \Phi(E(f)-E(f_{\infty})), \qquad \Phi(x)\ge0, \qquad \Phi(x) = 0 \iff x=0,
	$$ 
	for some nonnegative function $\Phi$.  
	Generally, such an inequality can only hold when all the conserved quantities are taken into account.
	If $\Phi'(0) \neq 0$, one usually gets exponential convergence toward
	$f_{\infty}$ in relative entropy $E(f)-E(f_{\infty})$ with a rate, which can be explicitly estimated. 
	\medskip
	
	The entropy method is a fully nonlinear alternative to arguments based on linearisation around the equilibrium and has the advantage of being quite robust with respect to variations and generalisations of the model system. 
	This is due to the fact that the entropy method  
	relies mainly on functional inequalities which have no direct link with the original PDE model.
	Generalised models typically feature related entropy and entropy-dissipation functionals and previously established entropy entropy-dissipation estimates may very usefully be re-applied.
	
	The entropy method has previously been used for scalar equations: nonlinear
	diffusion equations (such as fast diffusions \cite{CV,dolbeault}, Landau equation \cite{DVlandauII}),
	integral equations (such as the spatially homogeneous 
	Boltzmann equation \cite{toscani_villani1, toscani_villani2, villani_cerc}), 
	kinetic equations (see e.g. \cite{DVinhom1,DVinhom2, fell}), or coagulation-fragmentation equations (see e.g. \cite{CDF08,CDF08a}).
	For certain systems of drift-diffusion-reaction equations in semiconductor physics, an entropy entropy-dissipation estimate has been shown indirectly via a compact\-ness-based contradiction argument in \cite{GGH96,GH97,Gro92}.
	\medskip
	
	A first proof of entropy entropy-dissipation estimates for systems with explicit rates and constants was established in \cite{DeFe06, DeFe07, DeFe08} in the case of reversible reaction-diffusion equations. Recently, a new idea of proving entropy entropy-dissipation estimates in a general setting based on a convexification argument was presented in \cite{MHM}. 
	\medskip
	
	In this paper, we shall prove a new entropy entropy-dissipation estimate for the model system \eqref{e1}, which entails exponential convergence to equilibrium with explicitly computable constants and rates (see Theorem \ref{theo:Convergence} below).
	
	We remark two novelties:  i) this is (up to our knowledge) the first entropy entropy-dissipation estimate for a mixed volume-surface reaction-diffusion system, and ii) secondly, we introduce a new idea in the proof of entropy entropy-dissipation estimates for a system with general, superlinear, power-like nonlinearities, which we hope to turn out very useful when proving entropy entropy-dissipation estimates in more general settings. {This idea is indeed extended to more general systems which contain the system of this paper as a sub-case. Interested readers are referred to  \cite{DFT16,FT2015} for more details.}
	
	Moreover, we remark that, although the existence of weak solutions is obtained for general reaction rates $k_u$ and $k_v$, which can depend on time and space, we restrict for the sake of clarity the proof of explicit exponential convergence to equilibrium to the case of constant rates $k_u$ and $k_v$. The case of non-constant (in space and/or e.g. periodic in time) reactions rates leads to non-constant equilibria and requires a more involved formalism which can be treated in future works.
	\medskip
	
	We emphasise that we distinguish two cases in the equilibration analysis of \eqref{e1}: The non-degenerate diffusion case $\delta_v >0$ and the degenerate diffusion case $\delta_v = 0$. If $\delta_v >0$, then the surface diffusion term $-\delta_v\Delta_{\Gamma}v$ enables us to obtain an entropy entropy-dissipation estimate by only using the natural a-priori estimates derived from mass conservation, entropy and entropy-dissipation. 
	In the case of degenerate boundary diffusion $\delta_v = 0$, we derive an entropy entropy-dissipation estimate by using $L^{\infty}$ a-priori bounds of the solution. While such $L^{\infty}$-bounds can be shown to hold for the model  \eqref{e1}, they are often out of reach for more general systems with more concentrations in higher space dimensions, see e.g. \cite{CDF13}.
	However, we conjecture that in some (yet not all) cases of stoichiometric coefficients $\alpha, \beta$, the use of $L^{\infty}$-bounds should not be essential for the proof and could be avoided by more careful estimates. An example of such an estimate is presented in Proposition \ref{re:linear} when $\alpha = \beta = 1$.
	\medskip
	
	For future work, we hope that the robustness of the entropy method will enable us to study the 
	large time behaviour of more complicated and realistic models of asymmetric cell division by reusing  
	the entropy entropy-dissipation estimate derived in Theorem \ref{lem:E-EDEstimate} for the non-degenerate case $\delta_v >0$ and Lemma \ref{lem:degenerate_estimate} for the degenerate case $\delta_v =0$. 
	Thus, the considered mathematical core problem \eqref{e1} is also motivated by the goal of deriving 
	core entropy entropy-dissipation estimates, which encompasses the nonlinear boundary dynamics featured by the system \eqref{e1}.
	\medskip
	
	{ The rest of the paper is organised as follows. In Section 2, we show the global existence of a unique weak solution for system \eqref{e1} {under suitable assumptions on the reaction rates $k_u(t,x)$ and $k_v(t,x)$}. Parts of the proof of the existence theorem shall be detailed in the Appendix \ref{App}. Section 3 is devoted to the entropy method, establishing entropy entropy-dissipation estimates and proving explicit exponential convergence to equilibrium.}
	
	\section{Existence of a global solution}
	In this section, we will prove global existence of a unique weak solution to system \eqref{e1}. Though the proof is long and technical, the lines follow from a standard approach of upper and lower solutions. That is why we will state the main existence result in this section and leave the full proof to the Appendix \ref{App}.
	
	We define first our notion of weak solutions:
	\begin{definition}\label{def:weaksolution_1}
		A pair of functions $(u,v)$ is called a {\it weak solution} to system \eqref{e1} on $(0,T)$ if
		\begin{equation}
		u\in C([0,T];L^2(\Omega)), \ \ \text{ and }\ \ u\in L^{\infty}(0,T;L^{\infty}(\Omega))\cap L^{2}(0,T;H^1(\Omega)),
		\label{regular_u}
		\end{equation}
		\begin{equation}
		v\in C([0,T]; L^2(\Gamma)),\ \ \text{ and }\ \  v\in L^{\infty}(0,T;L^{\infty}(\Gamma))\cap L^2(0,T;H^1(\Gamma)),
		\label{regular_v}
		\end{equation}
		and the following weak formulation
		holds for all test functions $\varphi\in C^1([0,T];L^2(\Omega))\cap L^2(0,T;H^1(\Omega))$ and $\psi\in C^1([0,T];L^2(\Gamma))\cap L^2(0,T;H^1(\Gamma))$ with $\varphi \geq 0$, $\psi\geq 0$ and $\varphi(T) = \psi(T) = 0$:
		\begin{equation}\label{weakformulation}
		\begin{cases}
		\iTO[-u\varphi_t + \delta_u\nabla u\nabla \varphi]dxdt
		= \int_{\Omega}u_0\varphi(0)dx-\alpha\iTG(k_uu^{\alpha} - k_vv^{\beta})\varphi dSdt,\\[2mm]
		\iTG[-v\psi_t + \delta_v\nabla_{\Gamma}v\nabla_{\Gamma}\psi]dSdt
		= \int_{\Gamma}v_0\psi(0)dS + \beta\iTG(k_uu^{\alpha} - k_vv^{\beta})\psi dSdt,\end{cases}
		\end{equation}
		in which $\nabla_{\Gamma}$ is the so called tangential gradient on $\Gamma${, i.e. 
			$\nabla_{\Gamma} v = \nabla v - (\nu \cdot \nabla v ) \nu$ see e.g. \cite[(16.4) on page 389]{GT}.}
	\end{definition}
	
	\begin{remark}\label{remark:Definition} 
		With the regularity of $u$ and $v$ as stated in \eqref{regular_u} and \eqref{regular_v}, all left hand terms in \eqref{weakformulation} are clearly well defined. For the nonlinear reaction terms $\int_{\Gamma}k_uu^{\alpha}\varphi dS$ on the right hand side of \eqref{weakformulation}, we proceed as follows: First, if $u\in H^1(\Omega)\cap L^{\infty}(\Omega)$, we have
		\begin{equation*}
		\begin{aligned}
		\int_{\Gamma}|u|^{2\alpha}dx = \|u^{\alpha}\|_{L^2(\Gamma)}^2&\leq C(\|\nabla(u^\alpha)\|_{L^2(\Omega)}^2 + \|u^{\alpha}\|_{L^2(\Omega)}^2)\qquad  (\text{by using the Trace Theorem})\\
		&\leq C(\alpha^2\|u\|_{L^{\infty}(\Omega)}^{2\alpha-2}\|\nabla u\|_{L^2(\Omega)}^2 + |\Omega|\|u\|_{L^{\infty}(\Omega)}^{2\alpha}).
		\end{aligned}
		\end{equation*}
		Hence, $u^{\alpha}|_{\Gamma}\in L^2(\Gamma)$. Therefore 
		the weak formulation in Definition \ref{def:weaksolution_1} is well defined.
	\end{remark}
	\begin{definition}\label{aecomparision}
		We shall use the following shorthand notation 
		$
		(u_1, v_1)\geq (u_2, v_2)
		$ 
		for two pairs of functions $(u_1, v_1)$ and $(u_2, v_2)$ where $u_i(t,x): I\times \Omega \rightarrow \mathbb R$ and $v_i(t,x): I\times \Gamma\rightarrow \mathbb R, \; i = 1, 2, I\subset \mathbb R$, which means that 
		\begin{align*}
		u_1(t,x)\geq u_2(t,x)\quad \text{ a.e. in }\quad I\times \Omega, \qquad 
		v_1(t,x)\geq v_2(t,x)\quad \text{ a.e. in }\quad I\times \Gamma.
		\end{align*}
	\end{definition}
	For the sake of brevity we define the notation
	\begin{equation}\label{F}
	F(t,x,u,v):= -\alpha\bigl(k_u(t,x)u^{\alpha} - k_v(t,x)v^{\beta}\bigr), \qquad (t,x)\in[0,\infty)\times\Gamma,
	\end{equation}
	and
	\begin{equation}\label{G}
	G(t,x,u,v) := \beta\bigl(k_u(t,x)u^{\alpha} - k_v(t,x)v^{\beta}\bigr),\qquad (t,x)\in[0,\infty)\times\Gamma.
	\end{equation}
	
	The upper and lower weak solutions are defined as follows:
	\begin{definition}\label{def:UpperLowerSolution}
		A pair $(\overline{u}, \overline{v})$ is called an upper solution to the problem \eqref{e1} if $(\ou, \ov)$ satisfy the regularity \eqref{regular_u} and \eqref{regular_v} and that for all test functions $\varphi\in C^1(0,T;L^2(\Omega))\cap L^2(0,T;H^1(\Omega))$, $\psi\in C^1(0,T;L^2(\Gamma))\cap L^2(0,T;H^1(\Gamma))$ with $\varphi, \psi\geq 0$ and $\varphi(T) = \psi(T) = 0$, we have
		\begin{equation}
		\begin{cases}
		\iTO[-\ou\varphi_t +  \delta_u\nabla\ou\nabla\varphi]dxdt - \iTG F(t,x,\ou,\ov)\varphi dSdt \geq \int_{\Omega}\ou(0)\varphi(0) dx,\\[1mm]
		\iTG[-\ov\psi_t +  \delta_v\nabla_{\Gamma}\ov\nabla_{\Gamma}\psi]dSdt - \iTG G(t,x,\ou,\ov)\psi dSdt \geq \int_{\Gamma}\ov(0)\psi(0) dS,\\[1mm]
		\ou(0,x)\geq u_0(x) \text{ a.e. } x\in\Omega,\\[1mm]
		\ov(0,x)\geq v_0(x) \text{ a.e. }  x\in\Gamma.
		\end{cases}
		\label{a3}
		\end{equation}
		{ Lower solutions are defined in a similar way by replacing $\geq$ with $\leq$.}
	\end{definition}
	{
	In order to apply the method of upper and lower solutions, we need a comparison principle for system \eqref{e1} for pairs of upper and lower solutions. The following lemma can be proved using similar arguments as \cite{Souplet}:
	\begin{lemma}[Comparison Principle for Pairs of Upper and Lower Solutions]\label{lem:comparison}\hfill\\
		Let $0<T<\infty$. Assume $\underline{u}$, $\overline{u}$ satisfy \eqref{regular_u} and $\underline{v}$, $\overline{v}$ satisfy \eqref{regular_v}. Moreover, assume that for all testfunctions $\varphi\in C^1(0,T;L^2(\Omega))\cap L^2(0,T;H^1(\Omega)), \psi\in C^1(0,T;L^2(\Gamma))\cap L^2(0,T;H^1(\Gamma))$ with $\varphi, \psi\geq 0$ and $\varphi(T) = \psi(T) = 0$, we have
		\begin{equation}\label{e2}
		\begin{cases}
		\iTO[-(\uu - \ou)\varphi_t + \delta_u\nabla(\uu - \ou)\nabla \varphi]dxdt\\
		\qquad-\iTG (F(t,x,\uu,\uv) - F(t,x,\ou, \ov))\,\varphi \,dSdt \leq \int_{\Omega}(\uu(0)-\ou(0))\varphi(0) dx,\\[1mm]
		\iTG[-(\uv - \ov)\psi_t + \delta_v\nabla_{\Gamma}(\uv - \ov)\nabla_{\Gamma}\psi]dSdt\\
		\qquad- \iTG(G(t,x,\uu,\uv)-G(t,x,\ou,\ov))\,\psi\, dSdt\leq \int_{\Gamma}(\uv(0) - \ov(0))\psi(0) dS,\\[1mm]
		\uu(0,x)\leq \ou(0,x), \;x\in\Omega,\\[1mm]
		\uv(0,x)\leq \ov(0,x), \;x\in\Gamma.
		\end{cases}
		\end{equation}
		Then, $(\uu,\uv)\leq (\ou,\ov)$ in the sense of Definition \ref{aecomparision}.
	\end{lemma}}

	{
	Using the comparison for pairs of upper and lower solutions, we now can apply the method of converging sequences of solutions (see e.g. \cite{Pao}) to construct the solution to \eqref{e1} thanks to polynomial nonlinearities.
	
	\begin{theorem}\label{theo:ExistenceAndUniqueness}
		Let $\Omega\subset \mathbb R^n$ a bounded domain with smooth boundary $\Gamma = \partial\Omega$ (e.g. $\partial\Omega\in C^{2+\epsilon}$ with $\epsilon>0$). Let the diffusion coefficients $\delta_u>$ and $\delta_v\geq 0$, the stoichiometric coefficients $\alpha, \beta \in [1,+\infty)$. 
		Assume that the nonnegative reaction rate coefficients $k_u, k_v\in L^{\infty}([0,T]\times \Gamma)$ satisfy the bound
		\begin{equation}\label{bound_k}
			0< k_{\min} \leq k_u(t,x), k_v(t,x) \leq k_{\max} \quad \text{ for all } (t,x)\in [0,T]\times\Gamma.
		\end{equation}
		Moreover, we assume that the function
		\begin{equation}\label{pi}
		\pi(t,x) := \left(\frac{k_u(t,x)}{k_v(t,x)}\right)^{1/\beta} \quad \text{ for all } t>0, \quad x\in\Gamma,
		\end{equation}
is either constant, i.e. 
\begin{equation}\label{Assumpconst}
\pi(t,x) \equiv \pi >0 \quad \text{ for all } t>0, \quad x\in\Gamma,
\end{equation}		
or satisfies the inequality
		\begin{equation}\label{Assump}
		\partial_t\pi - \delta_v\Delta_{\Gamma}\pi \geq 0 \quad \text{ for all } t>0, \quad x\in\Gamma.
		\end{equation}
		Then, for all non-negative initial data $(u_0, v_0)\in L^{\infty}(\Omega)\times L^{\infty}(\Gamma)$, there exists a unique non-negative global weak solution $(u,v)$ for the system \eqref{e1}. 
	\end{theorem}
	}
	
\begin{remark}\label{existUpper}
		{The technical assumption \eqref{Assump} allows to generalise our proof for the existence of an upper solution to cases in which \eqref{Assumpconst} \emph{does not hold}. The question of whether this assumption is removable, is open for future investigation. Here, we give two examples where $k_u$ and $k_v$ satisfy the Assumption \eqref{Assump}. Note that 
		Assumption \eqref{Assumpconst} obviously implies Assumption \eqref{Assump}. }
		\begin{itemize}
			
			\item If $\pi(t,x) \equiv \pi(x)$, especially when $k_u$ and $k_v$ are time independent, then \eqref{Assump} holds whenever $\pi$ is a solution to the 
			homogeneous Laplace-Beltrami equation
			\begin{equation*}
			\Delta_{\Gamma}\pi(x) = 0, \quad \text{ for } x\in\Gamma.
			\end{equation*}
			Since $\Gamma$ is a smooth Riemannian manifold without boundary, the existence of such a $\pi$ is always guaranteed (see e.g. \cite[Chapter 5]{Taylorbook}). 
Some cases when $\pi(x)$ satisfies the inequality  
\begin{equation*}
-\Delta_{\Gamma}\pi(x) \ge 0, \quad \text{ for } x\in\Gamma.
\end{equation*}			
could be of biological interest.
For instance, by assuming $k_v$ constant, the above inequality is always satisfied provided 
$-\Delta_{\Gamma}k_u(x) \ge 0$, i.e that $k_u$ is proportional to a stationary state profile of a surface-diffusion process with non-negative source term. We emphasise, however, that 
for the Lgl model, we are unaware of such a biological mechanism.  Note also that the roles of $k_v$ and $k_u$
can be exchanged by swiching the role of $A$ and $B$ in the proof of Proposition \ref{UpperLower1}.

\item If $\pi(t,x) = T(t)X(x)$ where $T \in C^{1}([0,+\infty))$ is a nondecreasing function and $X(x)$ is a nonnegative solution to $\Delta_{\Gamma}X = 0$ on $\Gamma$, then \eqref{Assump} is fulfilled. Indeed, 
			$$
			\partial_t\pi(t,x) - \delta_{v}\Delta_{\Gamma}\pi(t,x) = X(x)T'(t) - \delta_vT(t)\Delta_{\Gamma}X = X(x)T'(t) \geq 0
			$$
			because $T$ is nondecreasing and $X$ is nonnegative.
		\end{itemize}
	\end{remark}
\begin{proof}
Since the proof of Theorem \ref{theo:ExistenceAndUniqueness} is rather lengthy, yet follows in essence the lines of e.g. \cite{Pao}, we postpone it to Appendix \ref{App} for the sake of readability.
\end{proof}	
	
	\section{Convergence to equilibrium}
	In this section, we assume that the reaction rates $k_u$ and $k_v$, and thus
	the equilibrium state $(u_{\infty}, v_{\infty})$ (see \eqref{f1_1} below), are constant. Moreover, for the sake of readability of the arguments, we shall assume normalised rates $k_u = k_v = 1$ (w.l.o.g. thanks to a rescaling in the cases $\alpha\neq\beta$). In any case, the following proofs can be readily generalised to arbitrary constants $k_u>0$, $k_v>0$. 
	
	We shall apply the entropy method to prove that the unique solution to \eqref{e1} converges exponentially fast to the equilibrium $(u_{\infty}, v_{\infty})$ for any initial data $(u_0,v_0)\in L^{\infty}(\Omega)\times L^{\infty}(\Gamma)$. 
	While the entropy method is certainly expected to apply to general reaction rates, the case of non-constant equilibria requires a more complicated formalism (see e.g. \cite{DFM}), 
	which we omit here for the sake of clarity of the argument and leave it for a future work. 
	
	In the following, we will first consider the non-degenerate case $\delta_{v}>0$ and later the degenerate case $\delta_v= 0$. We remark that in the first case with non-degenerate surface diffusion, our method relies only on natural a-priori bounds which are entailed by well-defined entropy and entropy-dissipation functionals along the flow of the solution. However,  in the case of degenerate diffusion, we require additional $L^{\infty}$-bounds of the solution. Since $L^{\infty}$-bounds of solutions for general systems are often unknown, the degenerate surface diffusion case poses more difficulties to be generalised than the non-degenerate case, which seems readily generalisable. 
	\medskip
	
	The system \eqref{e1} satisfies the {\it mass conservation law} \eqref{cons}, that is,
	\begin{equation*}
	M=\beta\int_{\Omega}u(t,x)dx + \alpha\int_{\Gamma}v(t,x)dS = \beta\int_{\Omega}u_0(x)dx + \alpha\int_{\Gamma}v_0(x)dS>0,
	\end{equation*}
	where we assume that the initial mass is positive ($M>0$). 
	
	The equilibrium of non-negative solutions of the system \eqref{e1} are the unique positive constants 
	$(u_{\infty}, v_{\infty})$, which balance the reaction rates, i.e. 
	\begin{equation}\label{f1}
	u_\infty^{\alpha} = v_\infty^{\beta},
	\end{equation}
	and satisfy the mass conservation law
	\begin{equation}
	\beta|\Omega|u_{\infty} + \alpha|\Gamma|v_{\infty} = M.
	\label{f1_0}
	\end{equation}
	We remark that the uniqueness of the equilibrium follows from the monotonicity of the 
	right hand sides of the equilibrium conditions 
	\begin{equation}
	u_{\infty}^{\alpha} = \Bigl(\frac{1}{\alpha|\Gamma|}(M-\beta|\Omega|u_{\infty})\Bigr)^{\beta},\qquad
	v_{\infty}^{\beta} = \Bigl(\frac{1}{\beta|\Omega|}(M-\alpha|\Gamma|v_{\infty})\Bigr)^{\alpha}
	\label{f1_1}
	\end{equation}
	on the intervals of equilibrium values, which are admissible for non-negative solutions of systems \eqref{e1}, i.e. $0 < u_{\infty} < \frac{M}{\beta|\Omega|}$ and $0<v_{\infty}< \frac{M}{\alpha|\Gamma|}$.
	\medskip
	
	As mentioned in the introduction, we prove the convergence to equilibrium by means of the entropy method. The method is based on the logarithmic entropy (free energy) functional
	\begin{equation}
	E(u, v) = \int_{\Omega}u(\log u - 1)dx + \int_{\Gamma}v(\log v - 1)dS
	\label{f2}
	\end{equation}
	and its non-negative entropy-dissipation
	\begin{equation}
	\begin{aligned}
	D(u, v) &= -\frac{d}{dt}E(u,v)\\
	&= \delta_u\int_{\Omega}\frac{|\nabla u|^2}{u}dx+ \delta_v\int_{\Gamma}\frac{|\nabla_{\Gamma}v|^2}{v}dS + \int_{\Gamma}(v^{\beta} - u^{\alpha})\log\frac{v^{\beta}}{u^{\alpha}}dS.
	\end{aligned}
	\label{f3}
	\end{equation}
	Our goal is to show that there exists a constant $C_0>0$ such that (see Theorem \ref{lem:E-EDEstimate} below)
	\begin{equation*}
	D(u,v) \geq C_0\left(E(u,v) - E(u_{\infty}, v_{\infty})\right)
	\end{equation*}
	for all non-negative $(u,v)$, which satisfy the mass conservation law \eqref{cons}. 
	Compared to previous related results on the entropy method for reaction-diffusion systems with quadratic nonlinearities (see \cite{DeFe06, DeFe07, DeFe08}), there are two main difficulties to overcome: the first is the treatment of the surface concentration $v$ and the associated boundary integrals and the second is the general nonlinear term $(v^{\beta} - u^{\alpha})\log\frac{v^{\beta}}{u^{\alpha}}$ for any $\alpha,\beta\ge1$. It is in particular the general nonlinearities, which necessitates a new proof compared to the quadratic nonlinearities considered in \cite{DeFe06, DeFe07, DeFe08}. We expect this new proof to constitute a more general approach. We refer to the preprint \cite{FT2015} for such a general approach.
	\medskip
	
	In the sequel, we will frequently use the following notations and inequalities:
	\begin{description}
		\item[Spatial averages and square-root abbreviation]
		\begin{align*}
		&\ou = \frac{1}{|\Omega|}\int_{\Omega}u\,dx, \qquad U = \sqrt{u},\qquad U_{\infty} = \sqrt{u_{\infty}}, 
		\qquad \overline{U} = \frac{1}{|\Omega|}\int_{\Omega}U\,dx,\\
		&\ov = \frac{1}{|\Gamma|}\int_{\Gamma}v\,dS,\qquad\, V = \sqrt{v}, \qquad\, V_{\infty} = \sqrt{v_{\infty}},\qquad
		\overline{V} = \frac{1}{|\Gamma|}\int_{\Gamma}V\,dS.
		\end{align*}
		\smallskip
		\item[Norms] $\|\cdot\|_{\Omega}$ and $\|\cdot\|_{\Gamma}$ are the norms in $L^2(\Omega)$ and $L^{2}(\Gamma)$ respectively. For a Banach space X, we denote by $\|\cdot\|_{X}$ its norm.
		\smallskip
		\item[Constants] A generic constant will be denoted by $C(M,\Omega,\dots)$ and may depend besides the arguments $M,\Omega,\dots$ also on $\alpha$ and $\beta$ 
		without explicitly stating the dependence on $\alpha$ and $\beta$. 
		Moreover, the constants $C_i(\dots)$ and $K_i(\dots)$ for $i=0,1,2,\dots$ are specific constants, for which the same rules of dependency hold.
		\smallskip
		\item[Inequalities] \hfill
		\begin{itemize}
			\item Poincare's inequality in $\Omega$
			\begin{equation*}
			P(\Omega)\int_{\Omega}|\nabla u|^2dx \geq \int_{\Omega}|u - \overline{u}|^2dx,
			\end{equation*}
			\item Poincare's inequality on $\Gamma$
			\begin{equation*}
			P(\Gamma)\int_{\Gamma}|\nabla_{\Gamma} v|^2dS \geq \int_{\Gamma}|v - \overline{v}|^2dS,
			\end{equation*}
			\item Trace Theorem
			\begin{equation}\label{Trace}
			T(\Omega)\int_{\Omega}|\nabla u|^2dx \geq \int_{\Gamma}|u - \overline{u}|^2dS.
			\end{equation}
		\end{itemize}
	\end{description}
	\medskip
	
	The mass conservation \eqref{cons} allows to rewrite the relative entropy towards the equilibrium 
	as  
	\begin{multline}
	E(u, v) - E(u_{\infty}, v_{\infty})
	= \int_{\Omega}u\log \frac{u}{\overline{u}}dx + \int_{\Gamma}v\log\frac{v}{\overline{v}}dS \\
	+\int_{\Omega}\Bigl(\overline{u}\log\frac{\overline{u}}{u_{\infty}} - (\overline{u} - u_{\infty})\Bigr)dx + 
	\int_{\Gamma}\Bigl(\overline{v}\log\frac{\overline{v}}{v_{\infty}} - (\overline{v} - v_{\infty})\Bigr)dS\\
	=\ I_1 + I_2,\label{f4} 
	\end{multline}
	where we define
	\begin{equation*}
	I_1 := \int_{\Omega}u\log \frac{u}{\overline{u}}\,dx + \int_{\Gamma}v\log\frac{v}{\overline{v}}\,dS,
	\end{equation*}
	and
	\begin{equation*}
	I_2 :=  \int_{\Omega}\Bigl(\overline{u}\log\frac{\overline{u}}{u_{\infty}} - (\overline{u} - u_{\infty})\Bigr)dx + 
	\int_{\Gamma}\Bigl(\overline{v}\log\frac{\overline{v}}{v_{\infty}} - (\overline{v} - v_{\infty})\Bigr)dS.
	\end{equation*}
	\medskip
	
	The following lemma proves, similarly to \cite{DeFe06}, a Csisz\'ar-Kullback-Pinsker type inequality, which quantifies that the relative entropy to equilibrium controls an $L^1$-distance:
	\begin{lemma}\label{lem:CP-Inequality}
		For all measurable functions $u: \Omega\rightarrow \mathbb R_{+}$ and $v:\Gamma \rightarrow \mathbb R_{+}$ satisfying 
		\begin{equation*}
		M=\beta\int_{\Omega}u\,dx + \alpha\int_{\Gamma}v\,dS >0,
		\end{equation*} 
		we have
		\begin{equation}\label{kk0}
		E(u, v) - E(u_{\infty}, v_{\infty}) \geq C_{\text{CKP}}\left(\|u-u_{\infty}\|_{L^1(\Omega)}^2 + \|v-v_{\infty}\|_{L^1(\Gamma)}^2\right),
		\end{equation}
		where $C_{\text{CKP}}>0$ is the following (non-optimal) constant depending only on the mass $M>0$ and $\alpha,\beta\ge1$:
		\begin{equation*}
		C_{\text{CKP}}
		= \frac{\min\left\{\alpha, \beta \right\}}{8M}.
		\end{equation*}

	\end{lemma}
	\begin{proof}
		By \eqref{f4}, we have that 
		\begin{equation*}
		E(u,v) - E(u_{\infty},v_{\infty}) = I_1 + I_2.
		\end{equation*}
		
		Considering the term $I_1$ at first, we use the classic Csisz\'ar-Kullback-Pinsker inequality (see e.g. \cite{Csi}) and the mass constraints 
		$\ou\le \frac{M}{\beta|\Omega|}$ and $\ov\le \frac{M}{\alpha|\Gamma|}$ to estimate 
		\begin{equation*}
		\int_{\Omega}u\log\frac{u}{\ou}dx \geq \frac{1}{2|\Omega|\ou}\|u - \ou\|_{L^1(\Omega)}^2 \geq \frac{\beta}{2M}\|u - \ou\|_{L^1(\Omega)}^2,
		\end{equation*}
		and
		\begin{equation*}
		\int_{\Gamma}v\log\frac{v}{\ov}dS \geq \frac{1}{2|\Gamma|\ov}\|v - \ov\|_{L^1(\Gamma)}^2 \geq \frac{\alpha}{2M}\|v - \ov\|_{L^1(\Gamma)}^2,
		\end{equation*}
		and, thus
		\begin{equation}
		I_1 \geq \frac{\beta}{2M}\|u - \ou\|_{L^1(\Omega)}^2 +  \frac{\alpha}{2M}\|v - \ov\|_{L^1(\Gamma)}^2.
		\label{kk1}
		\end{equation}
		
		Next, we rewrite $I_2$ in \eqref{f4} by introducing $q(x) = x\log x - x$, i.e.  
		\begin{equation*}
		I_2 = |\Omega|(q(\overline{u}) - q(u_{\infty})) + |\Gamma|(q(\overline{v}) - q(v_{\infty})),
		\end{equation*}
		where we have used that the mass conservation law \eqref{cons} implies 
		$$
		\int_{\Omega} (\ou-u_{\infty}) \log u_{\infty}\,dx + \int_{\Gamma} 
		(\ov-v_{\infty}) \log v_{\infty}\,dS=0
		$$ 
		since $ \frac{\log u_{\infty}}{\beta} =  \frac{\log u_{\infty}^{\alpha}}{\alpha\beta} = \frac{\log v_{\infty}^{\beta}}{\alpha\beta} =  \frac{\log v_{\infty}}{\alpha}$. Then, using again the conservation law \eqref{cons}, we denote \begin{equation*}
		Q(\ou) = |\Omega|q(\ou) + |\Gamma|\underbrace{q\biggl(\frac{M - \beta|\Omega|\ou}{\alpha|\Gamma|}\biggr)}_{=q(\overline{v})}\quad\text{and}\quad R(\ov) = |\Gamma|q(\ov) + |\Omega|\underbrace{q\biggl(\frac{M-\alpha|\Gamma|\ov}{\beta|\Omega|}\biggr)}_{=q(\overline{u})}.
		\end{equation*}
		Thus, we have the following two equivalent ways of writing $I_2$:
		\begin{align}
		I_2 &= Q(\ou) - Q(u_{\infty}) = R(\ov) - R(v_{\infty}).
		\label{kk2}
		\end{align}
		Moreover, direct computations give
		\begin{equation*}
		Q'(u_{\infty}) = |\Omega|q'(u_{\infty}) - \frac{\beta}{\alpha}|\Omega|q'\biggl(\frac{M-\beta|\Omega|u_{\infty}}{\alpha|\Gamma|}\biggr)
		= |\Omega|\log u_{\infty} - \frac{\beta}{\alpha}|\Omega|\log v_{\infty}= 0
		\end{equation*}
		since $u_{\infty}^{\alpha} = v_{\infty}^{\beta}$. Moreover, for any 
		$\ou_\theta$ satisfying the mass constraints 
		$0\le \ou_\theta\le \frac{M}{\beta|\Omega|}$
		\begin{align*}
		Q''(\ou_\theta) &= |\Omega|q''(\ou_\theta) + \frac{\beta^2}{\alpha^2}\frac{|\Omega|^2}{|\Gamma|}q''\biggl(\frac{M-\beta|\Omega|\ou_\theta}{\alpha|\Gamma|}\biggr)
		= |\Omega|\frac{1}{\ou_\theta} + \frac{\beta^2}{\alpha^2}\frac{|\Omega|^2}{|\Gamma|}\frac{\alpha|\Gamma|}{M-\beta|\Omega|\ou_\theta}\\
		&\geq \frac{\beta |\Omega|^2}{M}+\frac{\beta^2}{\alpha}\frac{|\Omega|^2}{M}=\frac{\beta}{\alpha}\frac{|\Omega|^2}{M}(\alpha + \beta).
		\end{align*}
		In a similar way, for any $0\le\ov_\theta\le \frac{M}{\alpha|\Gamma|}$, we estimate
		\begin{equation*}
		R'(v_{\infty}) = 0 \qquad\text{ and }\qquad R''(\ov_\theta) \geq \frac{\alpha}{\beta}\frac{|\Gamma|^2}{M}(\alpha+\beta).
		\end{equation*}
		Thus, altogether, Taylor expansion in \eqref{kk2} with $\ou_\theta=\theta\ou+(1-\theta)u_{\infty}$ and $\ov_\theta=\theta\ov+(1-\theta)v_{\infty}$ for some $\theta\in(0,1)$ yields
		\begin{multline}
		I_2= \frac{1}{2}(Q(\ou) - Q(u_{\infty})) + \frac{1}{2}(R(\ov) - R(v_{\infty})) \\
		\geq \frac{1}{4}\frac{\beta}{\alpha}\frac{|\Omega|^2}{M}(\alpha + \beta)(\ou - u_{\infty})^2 + \frac{1}{4}\frac{\alpha}{\beta}\frac{|\Gamma|^2}{M}(\alpha+\beta)(\ov - v_{\infty})^2\\
		= \frac{1}{4}\frac{\alpha +\beta}{M}\left(\frac{\beta}{\alpha}\|\ou - u_{\infty}\|_{L^1(\Omega)}^2 + \frac{\alpha}{\beta}\|\ov - v_{\infty}\|_{L^1(\Gamma)}^2\right).
		\label{kk3}
		\end{multline}
		Combining \eqref{kk1} and \eqref{kk3} with $\|u-\ou\|_{L^1(\Omega)}^2 + \|\ou - u_{\infty}\|_{L^1(\Omega)}^2 \ge \frac{1}{2}\|u - u_{\infty}\|_{L^1(\Omega)}^2$ by Jensen's inequality, we get 
		\begin{equation*}
		I_1+I_2 \ge \frac{\beta}{8M}\|u - u_{\infty}\|_{L^1(\Omega)}^2 + \frac{\alpha}{8M} \|v - v_{\infty}\|_{L^1(\Omega)}^2,
		\end{equation*}
		thus we obtain \eqref{kk0} with
		$C_{\text{CKP}} = \frac{\min\left\{\alpha, \beta \right\}}{8M}.
		$
	\end{proof}
	\medskip
	
	We now state our main result of this section, which is the exponential convergence to equilibrium with explicit rates and constants via the entropy method. The proof uses an entropy entropy-dissipation estimate, which is proven in Theorem \ref{lem:E-EDEstimate} below.  
	\begin{theorem}[Explicit Exponential Convergence to Equilibrium]\label{theo:Convergence}\hfill\\
		Assume that $\Omega\subset \mathbb R^n$ is a bounded domain with smooth boundary $\Gamma = \partial\Omega$ (e.g. $\partial\Omega\in C^{2+\epsilon}$ for any $\epsilon >0$). Then,  the unique weak solution $(u,v)$ of system \eqref{e1} subject to any nonnegative initial data $(u_0, v_{0})\in L^{\infty}(\Omega)\times L^{\infty}(\Gamma)$ satisfies the following exponential convergence to equilibrium
		\begin{equation}\label{kk4_0}
		\|u(t) - u_{\infty}\|_{L^{1}(\Omega)}^{2} + \|v(t) - v_{\infty}\|_{L^1(\Gamma)}^{2} \leq C_{\text{CKP}}^{-1}\,e^{-C_0t}\left(E(u_0, v_0) - E(u_{\infty}, v_{\infty})\right),
		\end{equation}
		where $C_0$ and $C_{\text{CKP}}^{-1}$ are positive constants as defined in Theorem \ref{lem:E-EDEstimate} below and Lemma \ref{lem:CP-Inequality} above and depend only on reaction rates $\alpha, \beta\ge1$, the diffusion rates $\delta_u>0$, $\delta_v\ge 0$, the domain $\Omega$, the boundary $\Gamma$ and the positive initial mass $M >0$.
	\end{theorem}
	\begin{proof}
		We have
		\begin{equation}\label{kk4}
		\frac{d}{dt}\left(E(u,v) - E(u_{\infty}, v_{\infty})\right) = \frac{d}{dt}E(u,v) = -D(u,v).
		\end{equation}
		On the other hand, by the Theorem \ref{lem:E-EDEstimate}, there exists $C_0>0$ such that
		\begin{equation}\label{kk5}
		D(u,v) \geq C_0\left(E(u,v) - E(u_{\infty}, v_{\infty})\right).
		\end{equation}
		Then, from \eqref{kk4}, \eqref{kk5} and the classical Gronwall inequality, we obtain
		\begin{equation}
		E(u(t), v(t)) - E(u_{\infty}, v_{\infty}) \leq e^{-C_0t}\left(E(u_0, v_0) - E(u_{\infty}, v_{\infty})\right).
		\label{kk6}
		\end{equation}
		Finally, the estimate \eqref{kk4_0} follows directly from \eqref{kk6} and Lemma \ref{lem:CP-Inequality}.
	\end{proof}
	
	\begin{remark}\label{remark:SystemInDomains}
		The techniques of this paper can be readily used to get the explicit exponential convergence to equilibrium for systems of the form:
		\begin{equation*}
		\begin{cases}
		u_t - d_u\Delta u = -\alpha(u^{\alpha} - v^{\beta}), &t>0,x\in\Omega,\\
		v_t - d_v\Delta v = \beta (u^{\alpha} - v^{\beta}), &t>0,x\in\Omega,\\
		\partial u/\partial \nu = \partial v/\partial \nu = 0, &t>0,x\in\partial\Omega,\\
		u(0,x) = u_0(x), v(0,x) = v_0(x), &x\in\Omega,
		\end{cases}
		\end{equation*}
		subject to non-negative initial data $u_0, v_0\in L^{\infty}(\Omega)$ and for all stoichiometric coefficients $\alpha, \beta \geq 1$ and positive diffusion coefficients $d_u, d_v$. By using Poincare's inequality $P(\Omega)\|\nabla v\|_{\Omega}^2 \geq \|v - \overline{v}\|_{\Omega}^2$ instead of the Trace inequality $T(\Omega)\|\nabla v\|_{\Omega}^2 \geq \|v - \overline{v}\|_{\Gamma}^2$, all the following arguments can be directly reproduced in the same way. Thus, the result of this paper, in a certain sense, completely solves the problem of trend to equilibrium for concentrations of the reversible chemical reaction of two species $\mathcal{U}$ and $\mathcal{V}$:
		\begin{figure}[htp]
			\begin{center}
				\begin{tikzpicture}
				\node (a) {$\alpha \, \mathcal{U}$}; \node (b) at (2,0) {$\beta\, \mathcal{V}$.};
				\draw[arrows=<-]  ([yshift=0.7mm]a.east) -- node [above] {\scalebox{.8}[.8]{}} ([yshift=0.7mm]b.west) ;
				\draw[arrows=<-] ([yshift=-0.7mm]b.west) -- node [below] {\scalebox{.8}[.8]{}} ([yshift=-0.7mm]a.east);
				\end{tikzpicture}
			\end{center}
		\end{figure}
	\end{remark}
	
	We shall now prove the key entropy entropy-dissipation estimate.
	
\begin{theorem}[Entropy Entropy-Dissipation Estimate]\label{lem:E-EDEstimate}\hfill\\
		{Assume that $\delta_u>0$ and $\delta_v\geq 0$.
		Consider measurable, non-negative functions $u: \Omega \rightarrow \mathbb R_{+}$ with trace $u|_{\Gamma}\in L^2(\Gamma)$ and $v: \Gamma \rightarrow \mathbb R_{+}$, which satisfy the mass conservation law
		\begin{equation}\label{EDDcons}
		\beta\int_{\Omega}u\,dx + \alpha\int_{\Gamma}v\,dS = M.
		\end{equation}
		In the case $\delta_v = 0$, we assume additionally that $(u,v) \leq (A,B)$ for two positive constants $A$ and $B$.}
		
		{Then, there exists a constant $C_0>0$ such that
			\begin{equation*}
				D(u,v) \geq C_0\left(E(u,v) - E(u_{\infty}, v_{\infty})\right),
			\end{equation*}
		where $C_0$ depends only on $M$, $|\Omega|$, $P(\Omega)$, $T(\Omega)$, 
		$|\Gamma|$, $P(\Gamma)$ as well as $\delta_u$, $\delta_v$, $\alpha$ and $\beta$, and also on $A$ and $B$ in the case $\delta_v = 0$.}
	\end{theorem}
	\noindent{\bf Proof of Theorem \ref{lem:E-EDEstimate}.} \hfill\\
	We divide the proof into two cases: $\delta_v>0$ in Section 3.1 and $\delta_v = 0$ in Section 3.2. 
	
	In the first case, we don't require any additional a-priori estimates on the solution besides well defined entropy and entropy-dissipation functionals in order to obtain the entropy-entropy dissipation estimate. 
	
	In the second case, since the diffusion term in $v$ is missing, we shall require a-priori $L^{\infty}$-bounds on the solution. However, we strongly believe that one might be able to avoid the use of $L^{\infty}$-bounds in some cases of the exponents $\alpha$ and $\beta$. 
	
	\subsection{The non-degenerate case: $\delta_{v} > 0$}\hfill \\ 
	We will show in the sequel that both $I_1$ and $I_2$ as defined in \eqref{f4} are bounded by the entropy dissipation. First, by using the Logarithmic-Sobolev inequality
	\begin{equation*}
	C_L(\Omega)\int_{\Omega}\frac{|\nabla u|^2}{u}dx \geq \int_{\Omega}u\log \frac{u}{\overline{u}}\,dx, \quad\text{ and }\quad C_L(\Gamma)\int_{\Gamma}\frac{|\nabla_{\Gamma}v|^2}{v}dS \geq \int_{\Gamma}v\log\frac{v}{\ov}\,dS,
	\end{equation*}
	we immediately get the following
	\begin{lemma}\label{lem:bound_I1}
		For all $t\geq 0$, we have
		\begin{equation}
		I_1 \leq C_2 \frac{D(u,v)}{2},
		\label{f5}
		\end{equation}
		where
		\begin{equation*}
		C_2 = {2}\max\left\{\frac{C_{L}(\Omega)}{\delta_u},\frac{C_{L}(\Gamma)}{ \delta_v}\right\}.
		\end{equation*}
	\end{lemma}
	\begin{remark}
		The factor ${2}$ in constant $C_2$ is chosen to still have $\frac{1}{2} D(u,v)$ left to estimate term $I_2$,  which is done in the following Lemma \ref{bound_I2}.
	\end{remark}
	
	\begin{lemma}\label{bound_I2}
		There exists $C_3>0$ such that, for all $t\geq 0$,
		\begin{equation}
		I_2 \leq C_3 \frac{D(u,v)}{2}.
		\label{f6}
		\end{equation}
	\end{lemma}
	\begin{proof}
		In a preliminary step, we observe that the function $\Phi: \mathbb R^2 \rightarrow \mathbb R$ defined by
		\begin{equation}\label{f5_0}
		\Phi(x,y) = \frac{x\log \frac{x}{y} - (x - y)}{(\sqrt{x}-\sqrt{y})^2}=\Phi\Bigl(\frac{x}{y},1\Bigr)
		\end{equation}
		{can be uniquely continuously extended onto $(0,\infty)^2$ by defining  
			$\Phi(y,y) := \lim_{x\to y} \Phi(\frac{x}{y},1)=2$ for all $y\in (0,+\infty)$, 
			see  \cite{DeFe06}. Moreover, for all $y\in(0,\infty)$, the function $\Phi(\cdot,y)$ is strictly increasing on $(0,\infty)$ and 
			satisfies $\lim\limits_{x\rightarrow 0}\Phi(x,y) = 1$. 
		}
		\medskip
		
		In a first step, we use now the mass conservation $\beta|\Omega|\overline{u} + \alpha|\Gamma|\overline{v} = M$
		to obtain the following bounds for $I_2$:
		\begin{equation}\label{f7}
		\int_{\Omega}\Bigl(\overline{u}\log\frac{\overline{u}}{u_\infty} - (\overline{u}-u_{\infty})\Bigr)\,dx \leq |\Omega|\,\Phi\biggl(\frac{M}{\beta|\Omega|}, u_{\infty}\biggr)\left(\sqrt{\overline{u}}-\sqrt{u_{\infty}}\right)^2
		\end{equation}
		and
		\begin{equation}
		\int_{\Gamma}\Bigl(\overline{v}\log\frac{\overline{v}}{v_\infty} - (\overline{v}-v_{\infty})\Bigr)\,dS \leq |\Gamma|\,\Phi\biggl(\frac{M}{\alpha|\Gamma|}, v_{\infty}\biggr)\left(\sqrt{\overline{v}} - \sqrt{v_{\infty}}\right)^2.
		\label{f8}
		\end{equation}
		Therefore, we have from \eqref{f7} and \eqref{f8} that
		\begin{equation}
		I_2 \leq K_0\left[\left(\sqrt{\overline{v}} - \sqrt{v_{\infty}}\right)^2+\left(\sqrt{\overline{u}}-\sqrt{u_{\infty}}\right)^2\right],
		\label{f8_1}
		\end{equation}
		where 
		$$
		K_0:=\max\left\{ |\Omega|\,\Phi\biggl(\frac{M}{\beta|\Omega|},u_{\infty}\biggr), |\Gamma|\,\Phi\biggl(\frac{M}{\alpha|\Gamma|}, v_{\infty}\biggr)\right\}.
		$$
		\smallskip
		
		Next, considering the entropy dissipation $D(u,v)$, we observe first that 
		\begin{equation}
		\delta_{u}\int_{\Omega}\frac{|\nabla u|^2}{u}dx = 4\delta_u\int_{\Omega}|\nabla \sqrt{u}|^2dx = 4\delta_u\|\nabla U\|_{\Omega}^2,
		\label{f9}
		\end{equation}
		and
		\begin{equation}
		\delta_{v}\int_{\Gamma}\frac{|\nabla_{\Gamma}v|^2}{v}dS = 4\delta_v\|\nabla_\Gamma v\|_{\Gamma}^2 \geq 4\delta_v\,P^{-1}(\Gamma)\|V - \overline{V}\|_{\Gamma}^2.
		\label{f10}
		\end{equation}
		Moreover, the elementary inequality $(a-b)\log\frac{a}{b} \geq 4(\sqrt{a}-\sqrt{b})^2$ yields
		\begin{equation}
		\int_{\Gamma}(v^{\beta} - u^{\alpha})\log\frac{v^{\beta}}{u^{\alpha}}dS \geq 4\|V^{\beta} - U^{\alpha}\|_{\Gamma}^2.
		\label{f11}
		\end{equation}
		Hence,
		\begin{equation}
		\frac{D(u,v)}{2}\geq 2\delta_u\|\nabla U\|_{\Omega}^2 + 2\delta_vP^{-1}(\Gamma)\|V - \overline{V}\|_{\Gamma}^2 + 2\|V^\beta - U^\alpha\|_{\Gamma}^2.
		\label{f11_1}
		\end{equation}
		Combining \eqref{f8_1} and \eqref{f11_1}, we see that in order to prove \eqref{f6} it is sufficient to find positive constants $K_1\le2$ and $K_2$ such that
		\begin{multline}
		2\delta_u\|\nabla U\|_{\Omega}^2+2\delta_vP^{-1}(\Gamma)\|V-\overline{V}\|_{\Gamma}^2 + K_1\|V^{\beta} - U^{\alpha}\|_{\Gamma}^2 \\\geq K_2 K_0\left[\Bigl(\sqrt{\overline{U^2}} - U_{\infty}\Bigr)^2 + \Bigl(\sqrt{\overline{V^{2}}} - V_{\infty}\Bigr)^2\right],
		\label{f12} 
		\end{multline}
		where we denote $\overline{U^2}=\frac{1}{|\Omega|}\int_{\Omega} U^2\,dx$ and $\overline{V^2}=\frac{1}{|\Gamma|}\int_{\Gamma} V^2\,dS$.
		
		\medskip
		
		In the following, we divide the proof of the key estimate \eqref{f12} into several steps. 
		As a preliminary remark, we recall
		that the estimate \eqref{f12} can only hold because of the constraint imposed by the conservation law \eqref{EDDcons} on $U$ and $V$, i.e. 
		\begin{equation}\label{conssqrt}
		\beta |\Omega| \overline{U^2} + \alpha |\Gamma| \overline{V^2} =M,
		\end{equation}
		since without \eqref{conssqrt}, the left hand side of \eqref{f12} vanishes for all constant states $U$, $V$ satisfying $V^{\beta}=U^{\alpha}$, while the right hand side of \eqref{f12} vanishes only at the equilibrium $U_{\infty}$, $V_{\infty}$.
		Thus, the following steps are designed as a chain of estimates, which allows for the conservation law \eqref{EDDcons} rewritten as \eqref{conssqrt} to enter into the proof of estimate \eqref{f12}.
		\medskip
		
		\noindent\underline{\it Step 1:} The goal of this step is to show that there exists a constant $K_3>0$ such that
		\begin{equation}
		\|V^{\beta} - U^{\alpha}\|_{\Gamma}^2 \geq \frac12\|\overline{V}^{\beta}-\overline{U}^{\alpha}\|_{\Gamma}^2 - K_3(\|U - \overline{U}\|_{\Gamma}^2 + \|V-\overline{V}\|_{\Gamma}^2).
		\label{f14}
		\end{equation}
		This inequality establishes a lower bound of the \emph{reaction entropy-dissipation term} in terms of a \emph{reaction entropy-dissipation term for the space averaged concentrations} $\overline{U}$ and $\overline{V}$ at the cost of 
		two terms, which can ultimately be controlled by the \emph{diffusion entropy-dissipation}. 
		
		At first, we remark that the averaged concentrations $\overline{U}$ and $\overline{V}$
		are bounded by Jensen's inequality and the conservation law \eqref{conssqrt}
		\begin{align}\label{boundU}
		\overline{U}^2 \le |\Omega|\overline{U^2} \le \frac{M}{\beta}\le\max\left\{1,\frac{M}{\beta}\right\}=:M_{\Omega},\\
		\overline{V}^2 \le |\Gamma|\overline{V^2} \le \frac{M}{\alpha}\le\max\left\{1,\frac{M}{\alpha}\right\}=:M_{\Gamma}.\label{boundV}
		\end{align}
		
		Next, we consider the following deviations around the spatially averaged concentrations:
		\begin{equation*}
		\delta_1(x) := U -\overline{U}, \quad \forall x\in\Omega, 
		\end{equation*}
		and 
		\begin{equation*}
		\delta_2(x) := V -\overline{V}, \quad \forall x\in\Gamma
		\end{equation*}
		and divide the boundary $\Gamma$ into two disjoint sets:
		\begin{equation*}
		\Gamma = S \cup S^{\perp},
		\end{equation*}
		where
		\begin{equation*}
		S:= \{x\in \Gamma: \ -\overline{U}\leq \delta_1(x) \leq \sqrt{M_{\Omega}},\ -\overline{V} \leq \delta_2(x)  \leq \sqrt{M_{\Gamma}}\}.
		\end{equation*}
		Note that $\delta_1\in   L^2(\Gamma)$ is well-defined by \eqref{f9} and the Trace Theorem \eqref{Trace}. 
		
		Due to the boundedness of  $\delta_1$ and $\delta_2$ in $S$, we readily estimate 
		by using Taylor expansion and Young's inequality
		\begin{multline}
		\|V^{\beta} - U^{\alpha}\|_{L^2(S)}^2= \|(\overline{V}+\delta_2)^{\beta} - (\overline{U}+\delta_1)^{\alpha}\|_{L^2(S)}^2\\ 
		\geq \frac12\|\overline{V}^{\beta} - \overline{U}^{\alpha}\|_{L^2(S)}^2 
		- \|\beta(\overline{V}+\theta_2)^{\beta-1}\delta_2 - \alpha(\overline{U}+\theta_1)^{\alpha-1}\delta_1\|_{L^2(S)}^2\\
		\geq \frac12\|\overline{V}^{\beta} - \overline{U}^{\alpha}\|_{L^2(S)}^2
		- C_3\bigl(M_{\Omega}^{\alpha-1},M_{\Gamma}^{\beta-1}\bigr)\left(\|\delta_1\|_{\Gamma}^2 +\|\delta_2\|_{\Gamma}^2\right),
		\label{b1}
		\end{multline}
		{where we have used that 
			$|\theta_1(x)|\le|\delta_1(x)|\le\sqrt{M_{\Omega}}$ and 
			$|\theta_2(x)|\le|\delta_2(x)|\le\sqrt{M_{\Gamma}}$ are bounded.}
		This proves \eqref{f14} on the set $S$.
		
		It remains to consider the set  
		\begin{equation*}
		S^{\perp} = \{x\in \Gamma: \ \delta_1(x) > \sqrt{M_{\Omega}} \quad \text{or} \quad \delta_2(x)>\sqrt{M_{\Gamma}}\}.
		\end{equation*}
		By using Chebyshev's inequality and by observing that for $\delta_1 > \sqrt{M_{\Omega}} \geq \overline{U}$, the set $\{x\in \Gamma:\delta_1^2 > M_{\Omega}\}$ coincides with the set $\{x\in \Gamma:\delta_1 > \sqrt{M_{\Omega}}\}$ and analog for $\delta_2 > \sqrt{M_{\Gamma}} \geq \overline{V}$, we get
		\begin{equation*}
		|\{x\in\Gamma : \delta_1 > \sqrt{M_{\Omega}}\}|=|\{x\in\Gamma : \delta_1^2\geq M_{\Omega}\}| \leq \frac{\|\delta_1\|_{\Gamma}^2}{M_{\Omega}},\end{equation*}
		and
		\begin{equation*}
		|\{x\in\Gamma : \delta_2 > \sqrt{M_{\Gamma}}\}|=|\{x\in\Gamma : \delta_2^2\geq M_{\Gamma}\}| \leq \frac{\|\delta_2\|_{\Gamma}^2}{M_{\Gamma}}.
		\end{equation*}
		Thus, it follows that 
		$$
		|S^{\perp}| \leq \frac{\|\delta_1\|_{\Gamma}^2}{M_{\Omega}} + 
		\frac{\|\delta_2\|_{\Gamma}^2}{M_{\Gamma}}.
		$$
		By the bounds \eqref{boundU}, \eqref{boundV}, we have moreover that $|\overline{V}^{\beta} - \overline{U}^{\alpha}|\leq C(M_{\Omega}^{\frac{\alpha}{2}},M_{\Gamma}^{\frac{\beta}{2}})$. Hence, since $M_{\Omega}\ge1$ and 
		$M_{\Gamma}\ge1$
		\begin{equation*}
		\|\overline{V}^{\beta} - \overline{U}^{\alpha}\|_{L^2(S^{\perp})}^2 \leq C(M_{\Omega}^{{\alpha}},M_{\Gamma}^{{\beta}})|S^{\perp}| \leq C(M_{\Omega}^{{\alpha}},M_{\Gamma}^{{\beta}})
		\left(\|\delta_1\|_{\Gamma}^2 + 
		\|\delta_2\|_{\Gamma}^2\right)
		\end{equation*}
		and, thus, 
		\begin{equation}
		\|V^{\beta} - U^{\alpha}\|_{L^2(S^{\perp})}^2 \geq 0\geq  \frac12\|\overline{V}^{\beta} - \overline{U}^{\alpha}\|_{L^2(S^{\perp})}^2 - C_4(M_{\Omega}^{{\alpha}},M_{\Gamma}^{{\beta}})(\|\delta_1\|_{\Gamma}^2 + \|\delta_2\|_{\Gamma}^2).
		\label{b2}
		\end{equation}
		Finally, the estimate \eqref{f14} is obtained from \eqref{b1} and \eqref{b2} for a constant $K_3(M_{\Omega}^{{\alpha}},M_{\Gamma}^{{\beta}})= C_3 + C_4$.
		\medskip
		
		With estimate \eqref{f14}, we proceed in estimating the left hand side of \eqref{f12}
		in the following way: We shall look for a positive constant $K_1\le2$ small enough, such that the following two conditions hold:
		\begin{equation*}
		\begin{cases}
		\delta_u\,T^{-1}(\Omega) - K_1K_3 \geq 0,\\
		\delta_vP^{-1}(\Gamma) - K_1K_3 \geq 0,
		\end{cases}
		\quad\Rightarrow\quad K_1\le\min\left\{\frac{\delta_u}{K_3 T(\Omega)},
		\frac{\delta_v}{K_3 P(\Gamma)},2\right\}.
		\end{equation*}
		Here, $T(\Omega)$ denotes the constant of the Trace inequality $T(\Omega)\|\nabla U\|_{\Omega}^2\ge \|U - \overline{U}\|_{\Gamma}^2$. We can then estimate the left hand side of \eqref{f12} by using \eqref{f14}
		\begin{align*}
		2\delta_u\|\nabla &U\|_{\Omega}^2 + 2\delta_vP^{-1}(\Gamma)\|V - \overline{V}\|_{\Gamma}^2 + K_1\|V^{\beta}-U^{\alpha}\|_{\Gamma}^2\nonumber\\ 
		&\geq \delta_u\|\nabla U\|_{\Omega}^2 + \delta_vP^{-1}(\Gamma)\|V - \overline{V}\|_{\Gamma}^2 + \frac{K_1}{2}\|\overline{V}^{\beta} - \overline{U}^{\alpha}\|_{\Gamma}^2\nonumber\\
		&\quad + (\delta_u\,T^{-1}(\Omega) - K_1K_3)\|U - \overline{U}\|_{\Gamma}^2 + (\delta_vP^{-1}(\Gamma) - K_1K_3)\|V - \overline{V}\|_{\Gamma}^2\nonumber\\
		&\geq \delta_u\|\nabla U\|_{\Omega}^2 + \delta_vP^{-1}(\Gamma)\|V - \overline{V}\|_{\Gamma}^2 + \frac{K_1}{2}\|\overline{V}^{\beta} - \overline{U}^{\alpha}\|_{\Gamma}^2.
		\end{align*}
		
		Therefore, in order to show \eqref{f12} it is sufficient to find suitable constants $K_4=\min\{\frac{2\delta_u}{K_1},\frac{2\delta_v}{K_1P(\Gamma)}\}$ and $K_5=\frac{2K_2 K_0}{K_1}$ in the following Step 2 such that: 
		\begin{equation}
		\|\overline{V}^{\beta}-\overline{U}^{\alpha}\|_{\Gamma}^2 + K_4(\|\nabla U\|_{\Omega}^2 + \|V-\overline{V}\|_{\Gamma}^2) \geq K_5\left[(\sqrt{\overline{U^2}} - U_{\infty})^2 + (\sqrt{\overline{V^{2}}} - V_{\infty})^2\right].
		\label{f15}
		\end{equation}
		\medskip
		
		\noindent\underline{\it Step 2:} To prove \eqref{f15}, we use the following change of variables with respect to the equilibrium 
		\begin{equation}
		\overline{U^2} = U_{\infty}^2(1+\mu_1)^2 \qquad\text{ and }\qquad \overline{V^2} = V_{\infty}^{2}(1+\mu_2)^2,
		\label{c1}
		\end{equation}
		which is well-adapted to the mass conservation law \eqref{conssqrt} in the sense that 
		\begin{equation}
		\beta|\Omega|U_{\infty}^{2}(1+\mu_1)^2 + \alpha|\Gamma|V_{\infty}^2(1+\mu_2)^2 = \beta|\Omega|U_{\infty}^2 + \alpha|\Gamma|V_{\infty}^2.
		\label{c2}
		\end{equation}
		From \eqref{c2}, it follows that the new variables $\mu_1$ and $\mu_2$ vary only in a bounded range of admissible values, i.e. $\mu_1\in [-1,+\mu_{1,m})$  and $\mu_2\in [-1,+\mu_{2,m})$, where a straightforward estimate shows $0<\mu_{1,m}<\frac{\alpha|\Gamma|V_{\infty}^2}{\beta|\Omega|U_{\infty}^2}$ and $0<\mu_{2,m}<\frac{\beta|\Omega|U_{\infty}^2}{\alpha|\Gamma|V_{\infty}^2}$.
		
		Moreover, equation \eqref{c2} implies that $\mu_1$ can be expressed  
		as a continuous, bounded function of $\mu_2$ (or the other way round), i.e.  
		\begin{equation}
		\mu_1(\mu_2) = -1 + \sqrt{1 - \frac{\alpha |\Gamma| V_{\infty}^2}{\beta |\Omega|U_{\infty}^2}(2\mu_2 + \mu_2^2)}
		= - R(\mu_2) \mu_2, 
		\label{c3}
		\end{equation}
		where
		\begin{equation*}
		R(\mu_2) :=  \frac{\frac{\alpha|\Gamma|V_{\infty}^2}{\beta|\Omega|U_{\infty}^2}(\mu_2+2)}{1 + \sqrt{1 - \frac{\alpha V_{\infty}^2|\Gamma|}{\beta U_{\infty}^2|\Omega|}(2\mu_2 + \mu_2^2)}}. 
		\end{equation*}
		We obviously have that $\mu_1(\mu_2=0)=0$, which represents the case $\overline{U^2} = U_{\infty}^2$ and $\overline{V^2} = V_{\infty}^2$.
		Moreover, $R(\mu_2)$ is a positive, monotone increasing function with 
		$$
		0<R(-1)=\frac{\frac{\alpha|\Gamma|V_{\infty}^2}{\beta|\Omega|U_{\infty}^2}}{1 + \sqrt{1 + \frac{\alpha |\Gamma|V_{\infty}^2}{\beta |\Omega|U_{\infty}^2}}}
		\le R(\mu_2) \le R(\mu_{2,m})< 2 \frac{\alpha|\Gamma|V_{\infty}^2}{\beta|\Omega|U_{\infty}^2} +1.
		$$
		Hence $R(\mu_2)$ for $\mu_2\in [-1,+\mu_{2,m})$ is uniformly bounded below and above by positive constants.
		\medskip
		
		Next, we notice that{
			\begin{equation*}
			\|\delta_1\|_{\Omega}^2 = \|U - \overline{U}\|_{\Omega}^2 = |\Omega|(\overline{U^2} - \overline{U}^2),
			\end{equation*}}
		and thus
		\begin{equation}
		{\overline{U} = \sqrt{\overline{U^2}} - \frac{1}{|\Omega|(\sqrt{\overline{U^2}}+\overline{U})}\|\delta_1\|_{\Omega}^2 = U_{\infty}(1+\mu_1) - \frac{1}{|\Omega|(\sqrt{\overline{U^2}}+\overline{U})}\|\delta_1\|_{\Omega}^2.
			\label{c6}}
		\end{equation}
		Similarly,
		\begin{equation}
		{\overline{V} = \sqrt{\overline{V^2}} - \frac{1}{|\Gamma|(\sqrt{\overline{V^2}}+\overline{V})}\|\delta_2\|_{\Gamma}^2 = V_{\infty}(1+\mu_2) - \frac{1}{|\Gamma|(\sqrt{\overline{V^2}}+\overline{V})}\|\delta_2\|_{\Gamma}^2.}
		\label{c7}
		\end{equation}
		We denote
		\begin{equation*}
		{R_1(U) := \frac{1}{|\Omega|(\sqrt{\overline{U^2}}+\overline{U})}\quad \text{ and } \quad R_1(V) := \frac{1}{|\Gamma|(\sqrt{\overline{V^2}}+\overline{V})}}
		\end{equation*}
		and remark that due to the lack of lower bounds for $\overline{U^2}\ge\overline{U}^2\ge0$ or $\overline{V^2}\ge\overline{V}^2\ge0$, we have no a-priori bounds to prevent $R_1(U)$ or $R_1(V)$ from being arbitrary large. Thus, we have to distinguish two cases, where the 
		first assumes a lower bound $\varepsilon>0$: 
		
		\noindent\underline{{\bf Case 1)} $\overline{U^2}\geq \varepsilon^2, \overline{V^2}\geq \varepsilon^2$:}\\[2mm]
		By \eqref{c6} and \eqref{c7}, the left hand side of \eqref{f15} is estimated as follows
		\begin{align}
		\|\overline{V}^{\beta} &- \overline{U}^{\alpha}\|_{\Gamma}^2 + K_4(\|\nabla U\|_{\Omega}^2 + \|V - \overline{V}\|_{\Gamma}^2)\nonumber\\
		&=\left\|(V_{\infty}(1+\mu_2) - R_1(V)\|\delta_2\|_{\Gamma}^2)^{\beta} - (U_{\infty}(1+\mu_1)- R_1(U)\|\delta_1\|_{\Omega}^2)^{\alpha}\right\|_{\Gamma}^2\nonumber\\
		&\quad + K_4(\|\nabla U\|_{\Omega}^2 + \|\delta_2\|_{\Gamma}^2)\nonumber\\
		&\geq |\Gamma|\left(V_{\infty}^{\beta}(1+\mu_2)^{\beta} - U_{\infty}^{\alpha}(1+\mu_1)^{\alpha}\right)^2 - C(\varepsilon^2, M)(\|\delta_2\|_{\Gamma}^2 + \frac{1}{P(\Omega)}\|\delta_1\|_{\Omega}^2)\nonumber\\
		&\quad+ K_4(\|\nabla U\|_{\Omega}^2 + \|\delta_2\|_{\Gamma}^2)\nonumber\\
		&\geq |\Gamma|\left(V_{\infty}^{\beta}(1+\mu_2)^{\beta} - U_{\infty}^{\alpha}(1+\mu_1)^{\alpha}\right)^2 - C(\varepsilon^2, M,\Omega)(\|\delta_2\|_{\Gamma}^2 + \|\nabla U\|_{\Omega}^2)\nonumber\\
		&\quad+ K_4(\|\nabla U\|_{\Omega}^2 + \|\delta_2\|_{\Gamma}^2)
		\label{c9}
		\end{align}
		by using the boundedness of $U_{\infty}$, $V_{\infty}$, $\|\delta_1\|_{\Omega}$, $\|\delta_2\|_{\Gamma}$, $\mu_1$, $\mu_2$, $R_1(U)$ and $R_1(V)$, the elementary inequality $(a - b)^2 \ge a^2/2 - b^2$, and by using Poincare's inequality. Choosing $K_4 \geq C(\varepsilon^2, M)$ in \eqref{c9} (by recalling that $K_4 = \min\{\frac{2\delta_u}{K_1}, \frac{2\delta_v}{K_1P(\Gamma)}\}$, this implies an additional constraint to choose $K_1$ small enough), we have
		\begin{equation}
		\|\overline{V}^{\beta} - \overline{U}^{\alpha}\|_{\Gamma}^2 + K_4(\|\nabla U\|_{\Omega}^2 + \|V - \overline{V}\|_{\Gamma}^2) \geq |\Gamma|(V_{\infty}^{\beta}(1+\mu_2)^{\beta} - U_{\infty}^{\alpha}(1+\mu_1)^{\alpha})^2.
		\label{c9_1}
		\end{equation}
		Therefore, in order to prove \eqref{f15}, it's enough to find $K_5$ such that
		\begin{equation*}
		|\Gamma|(V_{\infty}^{\beta}(1+\mu_2)^{\beta} - U_{\infty}^{\alpha}(1+\mu_1)^{\alpha})^2 \geq K_5\left(U_{\infty}^2\mu_1^2 + V_{\infty}\mu_2^2\right)
		\end{equation*}
		or equivalently,
		\begin{equation} 
		\frac{U_{\infty}^2\mu_1^2 + V_{\infty}^2\mu_2^2}{V_{\infty}^{2\beta}\left((1+\mu_2)^{\beta} - (1+\mu_1)^{\alpha}\right)^2} \leq \frac{|\Gamma|}{K_5}.
		\label{d1}
		\end{equation}
		
		In order to  estimate the denominator of \eqref{d1}, we consider the following two cases:\\
		\noindent{In the first case}, we assume that $-1\leq \mu_2 < 0$, from \eqref{c3} we have $\mu_1 > 0$. Then
		\begin{equation*}
		(1+\mu_2)^{\beta} \leq 1 + \mu_2< 1\quad \text{ and }\quad (1+\mu_1)^{\alpha} \geq 1 + \mu_1>1.
		\end{equation*}
		Hence,
		\begin{equation}
		|(1+\mu_2)^{\beta} - (1+\mu_1)^{\alpha}| \geq (1 + \mu_1) - (1 + \mu_2) = \mu_1- \mu_2=(1+R(\mu_2))|\mu_2|.
		\label{d3}
		\end{equation}
		
		\noindent{In the second case}, we consider $\mu_2 \geq 0$ and thus $\mu_1 \leq 0$ by \eqref{c3}. We estimate
		\begin{equation*}
		(1+\mu_2)^{\beta} \geq (1+\mu_2) \text{ and } (1+\mu_1)^{\alpha} \leq 1 + \mu_1,
		\end{equation*}
		and obtain therefore,
		\begin{equation}
		|(1+\mu_2)^{\beta} - (1+\mu_1)^{\alpha}| \geq (1 + \mu_2) - (1+\mu_1) = \mu_2-\mu_1=(1 + R(\mu_2))|\mu_2|.
		\label{d4}
		\end{equation}
		Altogether, \eqref{d3} and \eqref{d4} yield 
		\begin{equation}
		V_{\infty}^{2\beta}\left((1+\mu_2)^{\beta} - (1+\mu_1)^{\alpha}\right)^2 \geq V_{\infty}^{2\beta}(1 + R(\mu_2))^2\mu_2^2.
		\label{d5}
		\end{equation}
		
		For the numerator of \eqref{d1}, we use the expression \eqref{c3} to get
		\begin{equation}
		U_{\infty}^2\mu_1^2 + V_{\infty}^2\mu_2^2 = \left(V_{\infty}^2+ U_{\infty}^2\,R(\mu_2)^2\right)\mu_2^2, \label{d2}
		\end{equation}
		and combining \eqref{d2} and \eqref{d5} completes the proof of \eqref{d1}
		with a constant
		$$
		\frac{|\Gamma|}{K_5} \ge \frac{V_{\infty}^2+ U_{\infty}^2\,R(\mu_2)^2}{(1 + R(\mu_2))^2)V_{\infty}^{2\beta}}.
		$$
		
		Finally, by recalling that $K_5=\frac{2K_2 K_0}{K_1}$ and that $K_1$ was chosen small enough in the previous step, we conclude the first part of the proof of the Lemma by choosing 
		$K_2\le\frac{K_1 K_5}{2 K_0}$.
		\medskip
		
		\noindent\underline{{\bf Case 2)} $\overline{U^2}\leq \varepsilon^2 \text{ or } \overline{V^{2}}\leq \varepsilon^2$:}\\[2mm]
		For the second case, which considers states which are far away from the equilibrium $U\approx U_{\infty}$, $V\approx V_{\infty}$ for sufficiently small $\varepsilon$, we expect to be able to derive a positive lower bound for the entropy-dissipation in terms of $\varepsilon$. At first, we observe that the right hand side of \eqref{f15} is bounded by
		\begin{equation}
		K_5\left[(\sqrt{\overline{U^2}} - U_{\infty})^2 + (\sqrt{\overline{V^2}} - V_{\infty})\right]^2 \leq 2K_5(\overline{u} + \overline{v} + u_{\infty} + v_{\infty}) \leq K_5C(M).
		\label{c14}
		\end{equation}
		
		In the following, we consider two subcases of lower bounds of the entropy-dissipation. The first subcase 
		considers the situation where there is a lower bound of the diffusion entropy-dissipation since $U$ and $V$ 
		are not close to their spacial averages $\overline{U}$ and $\overline{V}$: \\
		\noindent\underline{Subcase 2.1) $\|\delta_1\|_{\Omega}^2 \geq \eta$ or $\|\delta_2\|_{\Gamma}^2 \geq \eta$:}\\
		By using Poincare's inequality $P(\Omega)\|\nabla U\|_{\Omega}^2 \geq \|\delta_1\|_{\Omega}^2$, we see that the left hand side of \eqref{f15} is bounded below by 
		\begin{equation}
		\begin{cases}
		K_4P^{-1}(\Omega)\eta &\text{ in the case } \|\delta_1\|_{\Omega}^2 \geq \eta,\\
		K_4\eta &\text{ in the case } \|\delta_2\|_{\Gamma}^2 \geq \eta.
		\end{cases}
		\label{c15}
		\end{equation}
		Thus, from \eqref{c14} and \eqref{c15}, we can obtain \eqref{f15} by choosing 
		\begin{equation*}
		{K_4 \geq K_5\max\left\{\frac{C(M)}{P(\Omega)\eta}, \frac{C(M)}{\eta}\right\}.}
		\end{equation*}
		
		\noindent\underline{Subcase 2.2) $\|\delta_1\|_{\Omega}^2 \leq \eta$ and $\|\delta_2\|_{\Gamma}^2 \leq \eta$:}\\
		This subcase concerns the situation where $U$ and $V$ are close to their spatial averages $\overline{U}$ and $\overline{V}$.
		Thus, since $U$ and $V$ are not close to the equilibrium $U_{\infty}$ and $V_{\infty}$ for sufficiently small $\varepsilon$ in {\bf Case 2)}, there has to be a lower bound for the reaction entropy-dissipation.
		
		Let us assume first $\overline{V^2} \leq \varepsilon^2$, thus $\overline{V}^{2}\leq \overline{V^2}\leq \varepsilon^2$. From
		\begin{equation*}
		\beta|\Omega|\overline{U^2} + \alpha|\Gamma|\overline{V^2} = M,\quad\text{and}\quad
		{\overline{U^2}  = \frac{\|\delta_1\|_{\Omega}^2}{|\Omega|} +\overline{U}^2,}
		\end{equation*}
		we estimate
		\begin{equation*}
		{\overline{U}^2 = \frac{1}{\beta|\Omega|}(M - \alpha|\Gamma|\overline{V^2}) - \frac{\|\delta_1\|_{\Omega}^2}{|\Omega|}
			\geq \frac{M}{\beta|\Omega|} - \frac{\alpha|\Gamma|}{\beta|\Omega|}\varepsilon^2 - \frac{\eta}{|\Omega|}.}
		\end{equation*}
		Hence, we can expand the reaction term by using $(a - b)^2 \geq a^2/2 - b^2$ as follows
		\begin{align}
		\|\overline{U}^{\alpha}-\overline{V}^{\beta}\|_{\Gamma}^2 &\geq |\Gamma|\left(\frac{1}{2}\overline{U}^{2\alpha} - \overline{V}^{2\beta}\right)
		{\geq |\Gamma|\left(\frac{1}{2}\left(\frac{M}{\beta|\Omega|} - \frac{\alpha|\Gamma|}{\beta|\Omega|}\varepsilon^2 - \frac{\eta}{|\Omega|}\right)^{\alpha} - \varepsilon^{2\beta}\right)}\nonumber\\
		&\geq \frac{|\Gamma|}{2^{\alpha+2}}\left(\frac{M}{\beta|\Omega|}\right)^{\alpha}
		\label{c17}
		\end{align}
		for small enough $\varepsilon$ and $\eta$.
		
		The case $\overline{U^2}\leq \varepsilon^2$ can be treated similarly and yields 
		\begin{equation}
		\|\overline{U}^{\alpha}-\overline{V}^{\beta}\|_{\Gamma}^2 \geq \frac{|\Gamma|}{2^{\beta+2}}\left(\frac{M}{\alpha|\Gamma|}\right)^{\beta}.
		\label{c17_1}
		\end{equation}
		
		From \eqref{c15}, \eqref{c17} and \eqref{c17_1}, we have for both cases $\overline{U^2}\leq \varepsilon^2$ or $\overline{V^2}\leq \varepsilon^2$ that the left hand side of \eqref{f15} is estimated below as
		\begin{multline}
		\|\overline{V}^{\beta} - \overline{U}^{\alpha}\|_{\Gamma}^2 + K_4(\|\nabla U\|_{\Omega}^2 + \|V - \overline{V}\|_{\Gamma}^2)\\
		\geq K_6= \min\left\{K_4P^{-1}(\Omega)\eta, K_4\eta, \frac{|\Gamma|}{2^{\alpha+2}}\left(\frac{M}{\beta|\Omega|}\right)^{\alpha}, \frac{|\Gamma|}{2^{\beta+2}}\left(\frac{M}{\alpha|\Gamma|}\right)^{\beta}\right\}.
		\label{c18_1}
		\end{multline}
		Then, \eqref{f15} follows from \eqref{c14}, \eqref{c18_1} by choosing $K_5 \leq \frac{K_6}{C(M)}$, which means to choose $K_2\le\frac{K_1 K_5}{2 K_0}$ small enough.
	\end{proof}
	
	\begin{remark}
		The Step 2 in the proof of Lemma \ref{bound_I2} can be significantly shortened if we consider the stoichiometric coefficients $\alpha\geq 2$ and $\beta\geq 2$, since we can prove \eqref{c9_1} without case distinction as follows.
		
		By recalling that $\|\delta_1\|_{\Omega}^2 = \|U-\overline{U}\|_{\Omega}^2=|\Omega|(\overline{U^2}-\overline{U}^2)$ and $\|\delta_2\|_{\Gamma}^2 = |\Gamma|(\overline{V^2} - \overline{V}^2)$, we derive the expressions
		\begin{equation*}
		\overline{U} 
		= \sqrt{\overline{U^2} - \|\delta_1\|_{\Omega}^2/|\Omega|}, \qquad
		\overline{V}
		= \sqrt{\overline{V^2} - \|\delta_2\|_{\Gamma}^2/|\Gamma|}.
		\label{c77}
		\end{equation*}
		Thus, by \eqref{c77}, we apply again Taylor expansion to estimate the first term on the left hand side of \eqref{f15} below by
		\begin{multline}\label{newtrick}
		\|\overline{V}^{\beta}-\overline{U}^{\alpha}\|_{\Gamma}^2=
		\left\|\left(\overline{V^2} - \|\delta_2\|_{\Gamma}^2/|\Gamma|\right)^{\frac{\beta}{2}} - \left(\overline{U^2}- \|\delta_1\|_{\Omega}^2/|\Omega|\right)^{\frac{\alpha}{2}}\right\|_{\Gamma}^2\\
		\geq \Bigl\|\overline{V^2}^{\frac{\beta}{2}} - \overline{U^2}^{\frac{\alpha}{2}}\Bigr\|_{\Gamma}^2
		- 2\int_{\Gamma} \left(\overline{V^2}^{\frac{\beta}{2}} - \overline{U^2}^{\frac{\alpha}{2}}\right) \left(\frac{\beta}{2} \Bigl(\overline{V^2}-\frac{\theta_2}{|\Gamma|}\Bigr)^{\!\frac{\beta}{2}-1}\frac{\|\delta_2\|_{\Gamma}^2}{|\Gamma|}\qquad\quad\right.\\
		\left.-\frac{\alpha}{2}\Bigl(\overline{U^2}-\frac{\theta_1}{|\Omega|}\Bigr)^{\!\frac{\alpha}{2}-1} \frac{\|\delta_1\|_{\Omega}^2}{|\Omega|}\right)dS
		\end{multline}
		for some $\theta_1/|\Omega|\le \|\delta_1\|_{\Omega}^2/|\Omega|\le\overline{U^2}\le M_{\Omega}$ and $\theta_2/|\Gamma|\le \|\delta_2\|_{\Gamma}^2/|\Gamma|\le\overline{V^2}\le M_{\Gamma}$. Note that $\frac{\beta}{2} - 1\geq 0$ and $\frac{\alpha}{2} - 1\geq 0$, then the last integral on the right hand side of \eqref{newtrick} can be estimated below by
		\begin{equation*}
		C(M_{\Omega}^{\alpha},M_{\Gamma}^{\beta},\Omega)\left(\frac{\|\delta_1\|_{\Omega}^2}{P(\Omega)}+\|\delta_2\|_{\Gamma}^2\right).
		\end{equation*}
		Thus, from \eqref{c1} and \eqref{newtrick}, we have 
		\begin{multline*}
		\|\overline{V}^{\beta}-\overline{U}^{\alpha}\|_{\Gamma}^2\\
		\geq |\Gamma|\left(V_{\infty}^{\beta}(1+\mu_2)^{\beta} - U_{\infty}^{\alpha}(1+\mu_1)^{\alpha}\right)^2
		- C(M_{\Omega}^{\alpha},M_{\Gamma}^{\beta},\Omega)\left(\frac{\|\delta_1\|_{\Omega}^2}{P(\Omega)}+\|\delta_2\|_{\Gamma}^2\right)\\
		\geq |\Gamma|\left(V_{\infty}^{\beta}(1+\mu_2)^{\beta} - U_{\infty}^{\alpha}(1+\mu_1)^{\alpha}\right)^2 - C(M_{\Omega}^{\alpha},M_{\Gamma}^{\beta},\Omega)\left(\|\nabla U\|_{\Omega}^2+\|\delta_2\|_{\Gamma}^2\right).
		\end{multline*}
		Therefore, by choosing $K_4 \geq C(M_{\Omega}^{\alpha},M_{\Gamma}^{\beta})$, we have proved \eqref{c9_1}:
		\begin{equation*}
		\|\overline{V}^{\beta}-\overline{U}^{\alpha}\|_{\Gamma}^2 + K_4(\|\nabla U\|_{\Omega}^2 + \|V - \overline{V}\|_{\Gamma}^2) \geq |\Gamma|\left(V_{\infty}^{\beta}(1+\mu_2)^{\beta} - U_{\infty}^{\alpha}(1+\mu_1)^{\alpha}\right)^2.
		\end{equation*}
		The rest of the proof follows exactly as the end of {\bf Case 1)} in Lemma \ref{bound_I2}.
	\end{remark}

	\subsection{The degenerate case: $\delta_{v} = 0$} 
	By Remark \ref{existUpper} and Proposition \ref{UpperLower1}, there exist two constants $A>0$ and $B>0$ such that $(A, B)$ is an upper solution to \eqref{e1}, then by the comparison principle we have that, for all $t\geq 0$,
	\begin{equation*}
	\|u(t)\|_{L^{\infty}(\Omega)} \leq A, \quad\text{ and }\quad \|v(t)\|_{L^{\infty}(\Gamma)} \leq B.
	\end{equation*}
	Hence the pair of constants $(A, B)$ in Theorem \ref{lem:E-EDEstimate} can be chosen as the above upper solution. By using the same function $\Phi$ as \eqref{f5_0}, we have
	\begin{align}
	E(u, v) &- E(u_{\infty}, v_{\infty})\nonumber\\
	&= \int_{\Omega}\left(u\log\frac{u}{u_{\infty}} - (u-u_{\infty})\right)dx + \int_{\Gamma}\left(v\log\frac{v}{v_{\infty}} - (v - v_{\infty})\right)dS\nonumber\\
	&\leq \Phi(A, u_{\infty})\int_{\Omega}(\sqrt{u} - \sqrt{u_{\infty}})^2dx + \Phi(B, v_{\infty})\int_{\Gamma}(\sqrt{v}-\sqrt{v_{\infty}})^2dS\nonumber\\
	&\leq \max\{\Phi(A, u_{\infty}), \Phi(B, v_{\infty})\}\left(\|U - U_{\infty}\|_{\Omega}^2 + \|V - V_{\infty}\|_{\Gamma}^2\right).
	\label{new0}
	\end{align}
	The following lemma, roughly speaking, shows that the diffusion of $u$ in $\Omega$ and the reversible reaction of $u$ and $v$ on $\Gamma$ lead to a diffusion-effect of $v$ on $\Gamma$:
	\begin{lemma}\label{lem:degenerate_estimate}
		There exists $C_1, C_2>0$ such that
		\begin{equation}
		C_1\|U^{\alpha}-V^{\beta}\|_{\Gamma}^2 + C_2\left(\|\nabla U\|_{\Omega}^2 + \|U - \overline{U}\|_{\Gamma}^2\right)  \geq C_3\|V-\overline{V}\|_{\Gamma}^2.
		\label{z1}
		\end{equation}
	\end{lemma}
	\begin{proof}
		Note that, by the Trace Theorem $T(\Omega)\|\nabla U\|_{\Omega}^2 \geq \|U - \overline{U}\|_{\Gamma}^2$, we could neglect the term $\|U - \overline{U}\|_{\Gamma}^2$ in \eqref{z1}. We write it here for the sake of readability.
		
		We will prove the inequality \eqref{z1} by distinguishing cases:
		
		\noindent\underline{\it Case 1: $\overline{U}\geq \varepsilon$.} Applying the ansatz
		\begin{equation*}
		V(x) = \overline{U}^{\frac{\alpha}{\beta}}(1+\delta(x)), \qquad \delta(x) \in [-1,+\infty)\qquad \forall x\in\Gamma,
		\end{equation*}
		we get
		\begin{multline}
		C_1\|U^{\alpha}-V^{\beta}\|_{\Gamma}^2 = C_1\|U^{\alpha}-\overline{U}^{\alpha}\|_{\Gamma}^2 - 2C_1\int_{\Gamma}(U^{\alpha}-\overline{U}^{\alpha})\overline{U}^{\alpha}[(1+\delta)^{\beta}-1]dS \\ + C_1\overline{U}^{2\alpha}\|(1+\delta)^{\beta}-1\|_{\Gamma}^2.
		\label{z2}
		\end{multline}
		{Since $\|U\|_{L^{\infty}(\Gamma)}\leq \sqrt{A}$, we have}
		\begin{equation}
		\|U^{\alpha} - \overline{U}^{\alpha}\|_{\Gamma}^2 \leq C(A)\|U-\overline{U}\|_{\Gamma}^2.
		\label{z3}
		\end{equation}
		From \eqref{z2} and \eqref{z3}, we can estimate the left hand side of \eqref{z1} as follows
		\begin{multline}
		C_1\|U^{\alpha}-V^{\beta}\|_{\Gamma}^2+C_2\|U - \overline{U}\|_{\Gamma}^2
		\geq \left(C_1+\frac{C_2}{C(A)}\right)\|U^{\alpha}-\overline{U}^{\alpha}\|_{\Gamma}^2 \\- 2C_1\int_{\Gamma}(U^{\alpha}-\overline{U}^{\alpha})\overline{U}^{\alpha}[(1+\delta)^{\beta}-1]dS + C_1\overline{U}^{2\alpha}\|(1+\delta)^{\beta}-1\|_{\Gamma}^2\\
		\geq \frac{C_1C_2}{C_1C(A)+C_2}\overline{U}^{2\alpha}\|(1+\delta)^{\beta}-1\|_{\Gamma}^2,
		\label{z4}
		\end{multline}
		where we have used Young's inequality
		\begin{multline}
		2C_1\int_{\Gamma}(U^{\alpha}-\overline{U}^{\alpha})\overline{U}^{\alpha}[(1+\delta)^{\beta}-1]dS\\
		\leq \left(C_1+\frac{C_2}{C(A)}\right)\|U^{\alpha}-\overline{U}^{\alpha}\|_{\Gamma}^2 + \frac{C_1^2C(A)}{C_1C(A) + C_2}\overline{U}^{2\alpha}\|(1+\delta)^{\beta}-1\|_{\Gamma}^2.
		\label{z5}
		\end{multline}
		Next, we observe that the function $R(\delta):=\frac{(1+\delta)^{\beta}-1}{\delta}$ 
		is continuous on $\delta\in[-1,\infty)$ with $R(0)=\beta\ge1$ and bounded below by $R(\delta)\ge R(-1)=1$ for $\delta\in[-1,\infty)$. Thus, 
		\begin{equation}\label{z6}
		\|(1+\delta)^{\beta}-1\|_{\Gamma}^2 = \int_{\Gamma} R(\delta)^2 \delta^2\,dS \ge 
		\int_{\Gamma} \delta^2\,dS.
		\end{equation}
		
		On the other hand, we have
		\begin{align}
		\|V - \overline{V}\|_{\Gamma}^2&= |\Gamma|\left(\overline{V^2}- \overline{V}^2\right)
		=  |\Gamma|\overline{U}^{\frac{2\alpha}{\beta}}\left(\overline{(1+\delta)^2}- \overline{1+\delta}^2\right)\nonumber\\
		&= |\Gamma|\overline{U}^{\frac{2\alpha}{\beta}} \left(1+2\overline{\delta}+\overline{\delta^2}-(1+\overline{\delta})^2\right)
		\le |\Gamma|\overline{U}^{\frac{2\alpha}{\beta}}\overline{\delta^2} \nonumber\\
		&\le \overline{U}^{\frac{2\alpha}{\beta}} \int_{\Gamma} \delta^2\,dS.\label{z7}
		\end{align}
		Now, keeping in mind that $\overline{U}\geq \varepsilon$, we obtain \eqref{z1} from \eqref{z4}, \eqref{z6} and \eqref{z7}, by choosing 
		\begin{equation*}
		C_3 \leq \frac{C_1C_2}{C_1C(M)+C_2}\min\{1; \varepsilon^{2\alpha(1-1/\beta)}\}.
		\end{equation*}
		
		\noindent\underline{\it Case 2: $\overline{U}\leq \varepsilon$.}
		We begin by considering $\overline{U}\le\varepsilon$, for which the contribution of $\|U-\overline{U}\|_{\Gamma}^2$ in \eqref{z1} can be arbitrary small when $U$ is close to $\overline{U}$. However, for $\varepsilon$ sufficiently small, we shall show that the estimate \eqref{z1} still holds because the reaction term 
		$\|U^{\alpha}-V^{\beta}\|_{\Gamma}^2$ can only be "small" if the $\|V-\overline{V}\|_{\Gamma}^2$ is of the "same order of smallness".
		
		We will treat two subcases: a) $\overline{U^2}$ is "small" and b) $\overline{U^2}$ is "big".\\[2mm]
		\noindent\underline{\it Case 2a): $\overline{U^2}\leq \frac{M}{2\beta|\Omega|}$.}
		A direct consequence of the conservation law \eqref{conssqrt} yields
		\begin{equation}\label{z10_0}
		\overline{V^2} = \frac{1}{\alpha|\Gamma|}\left(M-\beta|\Omega|\overline{U^2}\right)\geq \frac{M}{2\alpha|\Gamma|}.
		\end{equation}
		Next, we estimate the left hand side of \eqref{z1} similarly to \eqref{z2}--\eqref{z5} as
		\begin{equation*}
		C_1\|U^{\alpha}-V^{\beta}\|_{\Gamma}^2 + C_2\|U-\overline{U}\|_{\Gamma}^2 \geq C_4\|V^{\beta}-\overline{U}^{\alpha}\|_{\Gamma}^2.
		\end{equation*}
		where $C_4 = \frac{C_1C_2}{C_1C(M)+C_2}$. 
		Then, since $\overline{U}\le\varepsilon$
		\begin{equation*}
		C_4\|V^{\beta} - \overline{U}^{\alpha}\|_{\Gamma}^2 
		\geq C_4\int_{\Gamma}V^{2\beta}dS - 2C_4\varepsilon^{\alpha}\int_{\Gamma}V^{\beta}dS.
		\end{equation*}
		On the other hand, the right hand side of \eqref{z1} is bounded by
		\begin{equation}
		C_3\int_{\Gamma}|V - \overline{V}|_{\Gamma}^2 = C_3|\Gamma|\left(\overline{V^2} - \overline{V}^2\right) \leq C_3|\Gamma|\overline{V^2}.
		\label{z11_1}
		\end{equation}
		Therefore, in order to obtain \eqref{z1}, it is sufficient to prove that
		\begin{equation}
		C_4\int_{\Gamma}V^{2\beta}dS - 2C_4\varepsilon^{\alpha}\int_{\Gamma}V^{\beta}dS \geq C_3|\Gamma|\overline{V^2}.
		\label{z11_2}
		\end{equation}
		
		\noindent\underline{If $\beta \in [1,2]$}, by using Jensen's inequality (and noting that the function $f(x) = x^{\frac{\beta}{2}}$ is concave), we can estimate the left hand side of \eqref{z11_2} as
		{\begin{multline}
			C_4\int_{\Gamma}V^{2\beta}dS - 2C_4\varepsilon^{\alpha}\int_{\Gamma}V^{\beta}dS\\
			\geq C_4|\Gamma|^{1-\beta}\left(\int_{\Gamma}V^2dS\right)^{\beta} - 2C_4\varepsilon^{\alpha}|\Gamma|^{1-\beta/2}\left(\int_{\Gamma}V^2dS\right)^{\beta/2}\\
			= C_4|\Gamma|\overline{V^{2}}^{\beta} - 2C_4\varepsilon^{\alpha}|\Gamma|\overline{V^2}^{\beta/2}.
			\label{z12}
			\end{multline}}
		Since $\overline{V^2}\geq \frac{M}{2\alpha|\Gamma|}$, we can choose $\varepsilon$ and $C_3$ small enough such that
		\begin{equation*}
		{C_4\overline{V^2}^{\beta-1} \geq 2C_4\varepsilon^{\alpha}\overline{V^2}^{\beta/2-1} + C_3.}
		\end{equation*}
		After choosing 
		${\varepsilon\leq (\frac{1}{4})^{\frac{1}{\alpha}}(\frac{M}{2\alpha|\Gamma|})^{\frac{\beta}{2\alpha}}}$
		and ${C_3\leq \frac{C_4}{2}(\frac{M}{2\alpha|\Gamma|})^{\beta -1}}$, this gives together with \eqref{z12} the inequality \eqref{z11_2}.
		
		\noindent\underline{If $\beta \geq 2$}, by using Jensen's inequality, we have
		\begin{equation}
		{C_3|\Gamma|\overline{V^{2}} \leq C_3|\Gamma|\overline{V^{\beta}}^{2/\beta} \quad\text{ and } \quad C_4\int_{\Gamma}V^{2\beta}dS \geq C_4|\Gamma|\overline{V^{\beta}}^2.}
		\label{z14}
		\end{equation}
		Making use of \eqref{z14}, the relation \eqref{z11_2} can be proven provided
		\begin{equation*}
		{C_4|\Gamma|\overline{V^{\beta}}^2 - 2C_4\varepsilon^{\alpha}|\Gamma|\overline{V^{\beta}} \geq C_3|\Gamma|\overline{V^{\beta}}^{2/\beta}}
		\end{equation*}
		or equivalently
		\begin{equation*}
		{C_4\overline{V^{\beta}} \geq 2C_4\varepsilon^{\alpha} + C_3\overline{V^{\beta}}^{(2-\beta)/\beta}.}
		\end{equation*}
		This can be satisfied if we choose, for instance, 
		{$\varepsilon \le \frac{1}{4^{1/\alpha}}\bigl(\frac{M}{2\alpha|\Gamma|}\bigr)^{\beta/2\alpha}$ 
			and $C_3\leq \frac{1}{2}C_2\bigl(\frac{M}{2\alpha|\Gamma|}\bigr)^{\beta -1}$ and keeping in mind that $\overline{V^2}\ge\frac{M}{2\alpha|\Gamma|}$ and $\beta \geq 2$.}
		\medskip
		
		\noindent\underline{\it Case 2b): $\overline{U^2}\geq \frac{M}{2\beta|\Omega|}$.}
		Similarly to \eqref{z10_0}, we deduce from the conservation of mass that $\overline{V^2} \leq M/(2\alpha|\Gamma|)$. We estimate 
		$$
		C_2\|\nabla U\|_{\Omega}^2\ge \frac{C_2}{P(\Omega)} \|U-\overline{U}\|_{\Omega}^2\ge \frac{C_2|\Omega|}{P(\Omega)}\left(\overline{U^2}-\overline{U}^2\right) \ge
		\frac{C_2}{P(\Omega)}\left(\frac{M}{2\beta}-\varepsilon^2\right)
		\ge \frac{C_2}{P(\Omega)}\frac{M}{4\beta},
		$$
		if we chose $\varepsilon^2<\frac{M}{4\beta}$.
		Next, recalling \eqref{z11_1}, we estimate
		$$
		\frac{C_2}{P(\Omega)}\frac{M}{4\beta}
		\ge \frac{C_2}{P(\Omega)}\frac{M}{4\beta} \frac{\overline{V^2}}{M_{\Gamma}}\ge  C_3\|V-\overline{V}\|^2_{\Gamma},
		$$ 
		if we choose $C_3\le \frac{C_2\alpha|\Gamma|}{2\beta P(\Omega)}$.
		
		Altogether, the proof of \eqref{z1} is complete by choosing $\varepsilon$ and $C_3$ small enough in order to satisfy 
		the various constraints from the above cases.
	\end{proof}
	We are now ready to prove Theorem \ref{lem:E-EDEstimate} for degenerate case, that is when $\delta_v=0$, there exists $C_0>0$ such that
	\begin{equation*}
	D(u,v) \geq C_0(E(u,v) - E(u_{\infty}, v_{\infty})).
	\end{equation*}
	\begin{proof}
		We begin by estimating $D(u,v)$ below, that is 
		\begin{align*}
		D(u,v) &= \delta_{u}\int_{\Omega}\frac{|\nabla u|^2}{u}dx + \int_{\Gamma}(u^{\alpha}-v^{\beta})\log\frac{u^{\alpha}}{v^{\beta}}dS\\
		&\geq 4\delta_{u}\|\nabla U\|_{\Omega}^2 + 4\|U^{\alpha}-V^{\beta}\|_{\Gamma}^2,
		\end{align*}
		where we have used the elementary inequality $(a-b)\log(a/b)\geq 4(\sqrt{a} - \sqrt{b})^2$. 
		Then, by applying the Trace inequality
		$\|U - \overline{U}\|_{\Gamma}^2 \leq \|\nabla U\|_\Omega^2T(\Omega)$
		and Lemma \ref{lem:degenerate_estimate}, we get
		\begin{align}
		D(u,v)&\geq 4\delta_u\|\nabla U\|_{\Omega}^2 + 4\|U^{\alpha}-V^{\beta}\|_{\Gamma}^2\nonumber\\ 
		&\geq \theta\left[C_1\|U^{\alpha}-V^{\beta}\|_{\Gamma}^2 + C_2(\|\nabla U\|_{\Omega}^2 + \|U-\overline{U}\|_{\Gamma}^2)\right]\nonumber\\
		&\quad+ \left[4\delta_{u}-\theta C_2(1+T(\Omega))\right]\|\nabla U\|_{\Omega}^2 + (4-\theta C_1)\|U^{\alpha}-V^{\beta}\|_{\Gamma}^2\nonumber\\
		&\geq \theta C_3\|V - \overline{V}\|_{\Gamma}^2 + [4\delta_{u}-\theta C_2(1+T(\Omega))]\|\nabla U\|_{\Omega}^2
		+ (4-\theta C_1)\|U^{\alpha}-V^{\beta}\|_{\Gamma}^2\nonumber\\
		&\geq C_4\|\nabla U\|_{\Omega}^2 + C_5\|V-\overline{V}\|_{\Gamma}^2 + C_6\|U^{\alpha}-V^{\beta}\|_{\Gamma}^2
		\label{new2}
		\end{align}
		where we denote $C_4 = 4\delta_u - \theta C_2(1+T(\Omega))$, $C_5 = \theta C_3$ and $C_6 = (4-\theta C_1)$, where 
		$\theta>0$ is chosen such that the constants $C_4$ and $C_6$ are positive.
		
		In the following, we estimate the relative entropy $E(u,v)-E(u_{\infty}, v_{\infty})$ above by using \eqref{new0} 
	\begin{multline}
		E(u,v)-E(u_{\infty},v_{\infty})
		\leq \max\{\Phi(A,u_{\infty}), \Phi(B,v_{\infty})\}(\|U-U_{\infty}\|_{\Omega}^2 + \|V-V_{\infty}\|_{\Gamma}^2)\\
		\le C_8(\|U-\overline{U}\|_{\Omega}^2 + \|V-\overline{V}\|_{\Gamma}^2 + \|\overline{U}-U_{\infty}\|_{\Omega}^2 + \|\overline{V}-V_{\infty}\|_{\Gamma}^2)
		\label{new3}
		\end{multline}
		with $C_8 =2 \max\{\Phi(A,u_{\infty}), \Phi(B,v_{\infty})\}$. By using \eqref{new3}, we continue to estimate \eqref{new2} below and obtain by using Poincar\'e's inequality, the Trace Theorem and for $0<\varepsilon<1$ to be chosen
		\begin{align}
		D(u,v)
		&\geq \frac{C_4}{2P(\Omega)}\|U-\overline{U}\|_{\Omega}^2 + \frac{C_4}{2T(\Omega)}\|U-\overline{U}\|_{\Gamma}^2 + C_5\|V-\overline{V}\|_{\Gamma}^2 + C_6\|U^{\alpha}-V^{\beta}\|_{\Gamma}^2\nonumber\\
		&\geq \varepsilon\min\left\{\frac{C_4}{2P(\Omega)}, C_5\right\}(\|U-\overline{U}\|_{\Omega}+ \|V-\overline{V}\|_{\Gamma}^2)\nonumber\\
		&\quad+\frac{C_4}{2T(\Omega)}\|U-\overline{U}\|_{\Gamma}^2 +C_6\|U^{\alpha}-V^{\beta}\|_{\Gamma}^2 + C_5(1-\varepsilon)\|V-\overline{V}\|_{\Gamma}^2\nonumber\\
		&\geq \varepsilon\min\left\{\frac{C_4}{2P(\Omega)}, C_5\right\}\!\biggl(\frac{E(u,v)\!-\!E(u_{\infty},v_{\infty})}{C_8} - \|\overline{U}-U_{\infty}\|_{\Omega}^2 - \|\overline{V}-V_{\infty}\|_{\Gamma}^2\!\biggr)\nonumber\\
		&\quad+\frac{C_4}{2T(\Omega)}\|U-\overline{U}\|_{\Gamma}^2 +C_6\|U^{\alpha}-V^{\beta}\|_{\Gamma}^2+ C_5(1-\varepsilon)\|V-\overline{V}\|_{\Gamma}^2\nonumber\\
		&\geq \frac{\varepsilon}{C_8}\min\left\{\frac{C_4}{2P(\Omega)}, C_5\right\}(E(u,v)-E(u_{\infty},v_{\infty}))\nonumber\\
		&\quad+\frac{C_4}{2T(\Omega)}\|U-\overline{U}\|_{\Gamma}^2 +C_6\|U^{\alpha}-V^{\beta}\|_{\Gamma}^2+ C_5(1-\varepsilon)\|V-\overline{V}\|_{\Gamma}^2\nonumber\\
		&\quad - \varepsilon\min\left\{\frac{C_4}{2P(\Omega)}, C_5\right\}\left(\|\overline{U}-U_{\infty}\|_{\Omega}^2 + \|\overline{V}-V_{\infty}\|_{\Gamma}^2\right).
		\label{new4}
		\end{align}
		Now, by applying \eqref{f12} with $C_5(1-\varepsilon)$ in place of $4\delta_v\,P^{-1}(\Gamma)$, we can find a positive constant $\varepsilon>0$ small enough such that
		\begin{multline}
		\frac{C_4}{2T(\Omega)}\|U-\overline{U}\|_{\Gamma}^2 +C_6\|U^{\alpha}-V^{\beta}\|_{\Gamma}^2+ C_5(1-\varepsilon)\|V-\overline{V}\|_{\Gamma}^2\\
		\geq \varepsilon\min\left\{\frac{C_4}{2P(\Omega)}, C_5\right\}\left(\|\overline{U}-U_{\infty}\|_{\Omega}^2 + \|\overline{V}-V_{\infty}\|_{\Gamma}^2\right)
		\label{new5}
		\end{multline}
		holds and we conclude from \eqref{new4} and \eqref{new5} that 
		\begin{equation*}
		D(u,v) \geq \frac{\varepsilon}{C_8}\min\left\{\frac{C_4}{2P(\Omega)}, C_5\right\}(E(u,v)-E(u_{\infty},v_{\infty})),
		\end{equation*}
		which finishes the proof of the Lemma in the case of degenerate diffusion $\delta_v=0$. 
	\end{proof}
	
	As we can see in the proof the degenerate case, we used $L^{\infty}$-bounds of the solution, which are usually unavailable for more general systems. However, we believe that in some cases of stoichiometric coefficients $\alpha$ and $\beta$, there will be a way to show the exponential convergence to equilibrium without using the $L^{\infty}$ bounds. As example, we show that it is possible for the linear case, that is $\alpha = \beta = 1$.
	\begin{proposition}\label{re:linear} 
		Assume that $\alpha = \beta = 1$ and $\delta_v = 0$. The solution to the system \eqref{e1}, which rewrites as
		\begin{equation}
		\begin{cases}
		u_t - \delta_u \Delta u = 0, &x\in\Omega,\\
		\delta_u\frac{\partial u}{\partial \nu} = - u + v, &x\in\Gamma,\\
		v_t = u - v, &x\in \Gamma,\\
		u(0,x) = u_0(x),  &x\in\Omega,\\
		v(0,x) = v_0(x), &x\in \Gamma,
		\end{cases}
		\label{h1}
		\end{equation}
		converges exponentially to equilibrium in $L^2(\Omega)\times L^2(\Gamma)$.
	\end{proposition}
	\begin{remark}
		Due to the lack of surface diffusion $\delta_v\Delta_{\Gamma}v$, when establishing an entropy entropy-dissipation estimate, we need to prove an inequality analogous to \eqref{z1}, that is
		\begin{equation}
		C_1\|U^{\alpha}-V^{\beta}\|_{\Gamma}^2 + C_2\left(\|\nabla U\|_{\Omega}^2 + \|U - \overline{U}\|_{\Gamma}^2\right)  \geq C_3\|V-\overline{V}\|_{\Gamma}^2.\label{zzz}
		\end{equation}
		The main point of Proposition \ref{re:linear} is that, thanks to the linearity of the system, we can use the quadratic structure of the entropy to prove the existence of such an estimate without using the $L^{\infty}$-bounds of the solution (see \eqref{1star} below). For general $\alpha$ and $\beta$ an estimate like \eqref{zzz} seems 
		highly unclear: consider for instance a state $V=U^{\frac{\alpha}{\beta}}$. Then, $\|U^\alpha-V^\beta\|_{\Gamma}=0$
		and the two remaining terms $\|\nabla U\|_{\Omega}^2$ and $\|U - \overline{U}\|_{\Gamma}^2$ on the left hand side of 
		\eqref{zzz} seem not strong enough to ensure 
		the integrability of $V$ for $\alpha \gg \beta$. Such cases remain open problems to be treated in a future work.
	\end{remark}
	\begin{proof}
		The unique equilibrium $(u_{\infty}, v_{\infty})$ satisfies
		\begin{equation}\label{h1_1}
		\begin{cases}
		u_{\infty} = v_{\infty},\\
		|\Omega|u_{\infty} + |\Gamma|v_{\infty} = M
		\end{cases}
		\end{equation}
		where
		\begin{equation*}
		M = \int_{\Omega}u_0(x)dx + \int_{\Gamma}v_0(x)dS
		\end{equation*}
		is the initial mass. It follows that 
		\begin{equation}\label{equi}
		u_{\infty} = v_{\infty} = \frac{M}{|\Gamma| + |\Omega|}.
		\end{equation}
		
		For the sake of simplicity, we consider the quadratic entropy (which is only an admissible entropy functional since \eqref{h1} is linear)
		\begin{equation}
		E(u,v) = \|u\|_{\Omega}^2 + \|v\|_{\Gamma}^2,
		\label{h2}
		\end{equation}
		its entropy-dissipation
		\begin{equation}
		D(u,v) = -\frac{d}{dt}E(u,v) = 2d_u\|\nabla u\|_{\Omega}^2 + 2\|u - v\|_{\Gamma}^2,
		\label{h3}
		\end{equation}
		and the relative entropy
		\begin{equation}
		\begin{aligned}
		E(u,v) - E(u_{\infty},v_{\infty}) &= \|u\|_{\Omega}^2 + \|v\|_{\Gamma}^2 - \|u_{\infty}\|_{\Omega}^2 - \|v_{\infty}\|_{\Gamma}^2.
		\end{aligned}
		\label{h4}
		\end{equation}
		Similarly to \eqref{f4}, we decompose the relative entropy as follow:
		\begin{equation*}
		\begin{aligned}
		E(u,v) - E(u_{\infty}, v_{\infty}) &= [E(u,v)-E(\ou, \ov)] + [E(\ou,\ov) - E(u_{\infty}, v_{\infty})]\\ 
		&= [\|u - \ou\|_{\Omega}^2 + \|v-\ov\|_{\Gamma}^2] + [\|\ou - u_{\infty}\|_{\Omega}^2 + \|\ov - v_{\infty}\|_{\Gamma}^2].
		\end{aligned}
		\end{equation*}
		In the spirit of Lemma \ref{lem:degenerate_estimate}, we want to have an estimate similar to \eqref{z1}:
		\begin{equation}
		C_1\|\nabla u\|_{\Omega}^2 + C_2\|u - v\|_{\Gamma}^2 \ge C_3\|v - \ov\|_{\Gamma}^2.
		\label{0star}
		\end{equation}
		This can be done by estimating 
		\begin{align}
		C_1&\|\nabla u\|_{\Omega}^2 + C_2\|u - v\|_{\Gamma}^2\geq \frac{C_1}{T(\Omega)}\|u - \ou\|_{\Gamma}^2 + C_2\|u - v\|_{\Gamma}^2\nonumber\\
		&\geq \frac{C_1C_2/T(\Omega)}{C_2 + C_1/T(\Omega)}\|\ou - v\|_{\Gamma}^2
		= \frac{C_1C_2/T(\Omega)}{C_2 + C_1/T(\Omega)}(\|v - \ov\|_{\Gamma}^2 + \|\ou - \ov\|_{\Gamma}^2),
		\label{1star}
		\end{align}
		where we have used that $\int_{\Gamma} (\overline{u}-\overline{v})(v-\overline{v})\,dS=0$.
		Now, we can proceed similarly to the degenerate case to get the exponential convergence to equilibrium. For completeness, we sketch the entropy method as follows:
		\begin{equation}
		\begin{aligned}
		D(u,v)&= 2d_u\|\nabla u\|_{\Omega}^2 + 2\|u - v\|_{\Gamma}^2\\ 
		&\geq \frac{d_u}{P(\Omega)}\|u - \ou\|_{\Omega}^2 + \theta\left(C_1\|\nabla u\|_{\Omega}^2 + C_2\|u - v\|_{\Gamma}^2\right)\\
		&\geq \frac{d_u}{P(\Omega)}\|u - \ou\|_{\Omega}^2 + \theta C_3\|v - \ov\|_{\Gamma}^2 + \theta C_3\|\ou - \ov\|_{\Gamma}^2.
		\end{aligned}
		\label{2star}
		\end{equation}
		By the mass conservation law
		$$
		|\Omega|\ou + |\Gamma|\ov = M
		$$
		and the definition of the constant equilibrium \eqref{h1_1}, we observe first
		that  
		\begin{equation}\label{step}
		|\Omega|(\overline{u}-u_{\infty})+ |\Gamma|(\overline{v}-v_{\infty}) = 0.
		\end{equation}
		We thus calculate with $u_{\infty} = v_{\infty}$ and \eqref{step}
		\begin{multline}
		\|\overline{u} - \overline{v}\|_{\Gamma}^2 = |\Gamma|(\overline{u} - u_{\infty} + v_{\infty} -\overline{v})^2 \\= |\Gamma|(\overline{u} - u_{\infty})^2 
		+ |\Omega|(\overline{u} - u_{\infty})^2 
		+ \frac{|\Gamma|}{|\Omega|}(\overline{v}-v_{\infty})^2 + |\Gamma|(\overline{v}-v_{\infty})^2\\
		= \frac{|\Gamma|+|\Omega|}{|\Omega|} \left(\|\overline{u} - u_{\infty}\|^2
		+ \|\overline{v}-v_{\infty}\|^2 \right)\label{3star}
		\end{multline}
		
		
		Combining \eqref{2star} and \eqref{3star} yields
		{\begin{align*}
			D(u,v) &\geq \frac{d_u}{P(\Omega)}\|u - \ou\|_{\Omega}^2 + \theta C_3\|v - \ov\|_{\Gamma}^2 \\
			&\quad+ \theta C_3\frac{|\Omega|+|\Gamma|}{|\Omega|}(\|\ou - u_{\infty}\|_{\Omega}^2 + \|\ov - v_{\infty}\|_{\Gamma}^2)\\
			&\geq C_4(E(u,v) - E(u_{\infty}, v_{\infty})).
			\end{align*}}
		Hence, the solution satisfies the exponential convergence to equilibrium:
		\begin{equation*}
		\|u - u_{\infty}\|_{\Omega}^2 + \|v-v_{\infty}\|_{\Gamma}^2 \leq e^{-C_4t}(\|u_0 - u_{\infty}\|_{\Omega}^2 + \|v_0 - v_{\infty}\|_{\Gamma}^2).
		\end{equation*}
	\end{proof}

	\section{Appendix}\label{App}
	In this appendix, we will give the full proof of Theorem \ref{theo:ExistenceAndUniqueness} by the technique of upper and lower solutions. 
	The following lemma is proved direct computations.
	\begin{lemma}\label{Nonlinearities}
		The functions $F$ and $G$ are defined in \eqref{F} and \eqref{G} respectively are locally Lipschitz. In particular, given a pair of non-negative functions $(\ou, \ov) \ge (0,0)$, 
					there exist two non-negative bounded functions $L_u(t, x)$, $L_v(t, x)\in L^{\infty}([0,\infty)\times\Gamma)$ such that, for all $(\ou,\ov)\geq (u_1,v_1),(u_2,v_2)\geq (0,0)$, the followings hold pointwise in $(t,x)\in[0,\infty)\times\Gamma:$
					\begin{align}
					F(t, x, u_1, v_1) - F(t, x, u_2, v_2) &\leq \alpha L_u(t,x)(u_2 - u_1)_{+}+\alpha L_v(t,x)(v_1 - v_2)_{+},
					\label{FLipschitzupper} \\
					F(t, x, u_1, v_1) - F(t, x, u_2, v_2) &\geq -\alpha L_u(t,x)(u_1 - u_2)_{+}-\alpha L_v(t,x)(v_2 - v_1)_{+},
					\label{FLipschitzlower}
					\end{align}
					and
					\begin{align}
					G(t,x,u_1,v_1) - G(t,x,u_2,v_2) &\leq \beta L_u(t,x)(u_1 - u_2)_{+}+\beta L_v(t,x)(v_2 - v_1)_{+},\label{GLipschitzupper}\\
					G(t,x,u_1,v_1) - G(t,x,u_2,v_2) &\geq -\beta L_u(t,x)(u_2 - u_1)_{+}-\beta L_v(t,x)(v_1 - v_2)_{+},\label{GLipschitzlower}
					\end{align}
					where $(\cdot)_{+}$ denotes the positive part, that is $(w)_{+} = w$ if $w\geq 0$ and $(w)_{+} = 0$ otherwise. 
	\end{lemma}

	\medskip
	
	{By subtracting the lower and upper solutions (as they are defined in Definition \ref{def:UpperLowerSolution}), the comparison Lemma \ref{lem:comparison} yields the following:}
	\begin{lemma}\label{lem:compare}
		If $(\ou, \ov)$ is an upper solution and $(\uu, \uv)$ is a lower solution to \eqref{e1}, then we have $(\ou,\ov)\geq (\uu,\uv)$ in the sense of Definition \ref{aecomparision}.
	\end{lemma}
	
	While $(\uu,\uv) = (0,0)$ is clearly a lower solution to system \eqref{e1}, the existence of an upper solution to system \eqref{e1}
	is in general a difficult question, which we are only able to answer partially by the following proposition under the technical assumption \eqref{Assump}
	on the reaction rates $k_u(t,x)$ and $k_v(t,x)$.

	\begin{proposition}\label{UpperLower1}
		The pair $(\uu,\uv) = (0,0)$ is a trivial lower solution to system \eqref{e1}. Moreover, if the function
		\begin{equation}\label{pi1}
		\pi(t,x):= \left(\frac{k_u(t,x)}{k_v(t,x)}\right)^{1/\beta} \qquad \text{for}\quad t>0, \; x\in\Gamma,
		\end{equation}
		satisfies either the assumption \eqref{Assumpconst} or the assumption \eqref{Assump},
		then the pair $(\ou, \ov) = (A, B(t,x))$ defined as 
		\begin{equation}\label{defineA1}
		A := \max\left\{\|u_0\|_{L^{\infty}(\Omega)}, \left(\frac{k_{\max}}{k_{\min}}\|v_0\|_{L^{\infty}(\Gamma)}^\beta\right)^{1/\alpha}\right\}
		\end{equation}
		and
		\begin{equation}\label{defineB1}
		B(t,x) := \pi(t,x)A^{\alpha/\beta} \qquad \text{for}\quad t>0, x\in\Gamma
		\end{equation}
		is an upper solution to system \eqref{e1}.
	\end{proposition}
	
	\begin{proof}
		Clearly $(\uu,\uv) = (0,0)$ is a lower solution. Assume that  \eqref{Assump} holds, which is also true 
		provided that  \eqref{Assumpconst} holds. 
		To verify that $(A, B(t,x))$ is an upper solution, we first observe that
		$$
		A \geq \|u_0\|_{L^{\infty}(\Omega)} \qquad \text{ and }\qquad B(0,x) \geq \left(\frac{k_{\min}}{k_{\max}}\right)^{1/\beta}A^{\alpha/\beta}\geq \|v_0\|_{L^{\infty}(\Gamma)} \quad \forall x\in\Gamma,
		$$
		which means that $(A, B(t,x))$ satisfies the last two conditions in \eqref{a3}. Next, we see that the reaction term vanishes,
		\begin{equation}\label{vanishRe}
		k_uA^{\alpha} - k_vB^{\beta}(t,x) = k_uA^{\alpha} - k_v[\pi(t,x)A^{\alpha/\beta}]^{\beta}= 
		0
		\end{equation}
		thanks to \eqref{pi1} and \eqref{defineB1}. Hence the constant $A$ satisfies the first equation in \eqref{a3}. From \eqref{defineB1} and the assumption \eqref{Assump}, it follows that
		\begin{equation*}
		\partial_tB(t,x) - \delta_v\Delta_{\Gamma}B(t,x) \geq 0 \qquad \text{for}\quad t>0, \; x\in\Gamma.
		\end{equation*}
		Thus, by testing this equation with a nonnegative test function $\psi$ satisfying $\psi(T)=0$, we obtain
		\begin{equation*}
		\iTG[-B\psi_t + \delta_v\nabla_{\Gamma}B\nabla_{\Gamma}\psi]dSdt \geq \int_{\Gamma}B(0)\psi(0)dS,
		\end{equation*}
		which in combination with the vanishing reaction term \eqref{vanishRe} shows that $B(t,x)$ satisfies the second equation in \eqref{a3}.
	\end{proof}


	In order to prove our existence result, we introduce the following auxiliary functions, which will be useful for proving the monotonicity of the sequences of upper and lower solutions
	\begin{align}
	f(t,x,u,v) &= F(t,x,u,v) + \alpha L_u(t,x)\,u \nonumber\\
	g(t,x,u,v) &= G(t,x,u,v) + \beta L_v(t,x)\,v.
	\label{fg}
	\end{align}
	where $L_u$ and $L_v$ are in Lemma \ref{Nonlinearities}.
	\begin{lemma}\label{newnonlinearities}
		The functions $f$ and $g$ inherit the following properties from the functions $F$ and $G$:
		\begin{itemize}
			\item[(i)] The functions $f(t,x,u, \cdot)$ and $g(t,x,\cdot,v)$ are non-decreasing for any $t,x\in \mathbb R\times \Gamma$ and any $u, v\in \mathbb R$.
			\item[(ii)] For all $(\ou,\ov)\geq (u_1,v_1)\geq (u_2,v_2)\geq (0,0)$, there holds:
			\begin{equation*}
			f(t, x, u_1, v) - f(t, x, u_2, v) \geq -\alpha L_u(t,x)(u_1 - u_2)_{+}+ \alpha L_u(t,x)(u_1-u_2) =  0,
			\end{equation*}
			and
			\begin{equation*}
			g(t,x,u,v_1) - g(t,x,u,v_2) \geq -\beta L_v(t,x)(v_1 - v_2)_{+}+\beta L_v(t,x)(v_1-v_2) =  0.
			\end{equation*}
			Thus, the functions $f(t,x,\cdot,v)$ and $g(t,x,u,\cdot)$ are monotone non-de\-creasing for all $(\ou,\ov)\geq (u_1,v_1)\geq (u_2,v_2)\geq (0,0)$ contrary to $F$ and $G$.
		\end{itemize}
	\end{lemma}
	\begin{proof}
		The statements of the above Lemma follow directly from Lemma \ref{Nonlinearities}, 
		in particular from \eqref{FLipschitzlower} and \eqref{GLipschitzlower}.
	\end{proof}
	\medskip
	With the notation \eqref{fg}, system \eqref{e1} rewrites as
	\begin{equation}
	\begin{cases}
	u_t - \delta_u\Delta u = 0, &x\in\Omega, t>0,\\
	\delta_u\frac{\partial u}{\partial \nu} + \alpha L_u(t,x)\,u = f(t,x,u,v),&x\in\Gamma, t>0,\\
	v_t - \delta_v\Delta_{\Gamma}v + \beta L_v(t,x)\,v= g(t,x,u,v),&x\in\Gamma, t>0.
	\end{cases}
	\label{a2_2}
	\end{equation}
	Hereafter,  we write $f(u,v)$ and $g(u,v)$ for $f(t,x,u,v)$ and $g(t,x,u,v)$, respectively, except where it is stated otherwise.
	
	\medskip
	
	Starting from the pair of lower and upper solutions $(\uu,\uv)\leq(\ou,\ov)$ 
	as constructed in Proposition \ref{existUpper}, we will construct a sequence of lower solutions $\{(\uu^{(k)}, \uv^{(k)})\}_{k\geq 0}$ as follows:
	\begin{equation}\tag{I.0}\label{I0}
	(\uu^{(0)}, \uv^{(0)}) = (\uu, \uv),
	\end{equation}
	and for all $k \geq 1$, $\uu^{(k)}$ and $\uv^{(k)}$ are the solutions of the following heat equation with inhomogeneous Robin boundary condition:
	\begin{equation}\tag{I.1}\label{I1}
	\begin{cases}
	\partial_t\underline{u}^{(k)} - \delta_u\Delta\underline{u}^{(k)} = 0, &x\in\Omega, t>0,\\
	\delta_u\frac{\partial\uu^{(k)}}{\partial\nu} + \alpha L_u\uu^{(k)} = f(\uu^{(k-1)}, \uv^{(k-1)}), &x\in\Gamma, t>0,\\
	\uu^{(k)}(0,x) = u_0(x) \in L^{\infty}(\Omega),&x\in\Omega,
	\end{cases}
	\end{equation}
	and the following linear inhomogeneous equation:
	\begin{equation}\tag{I.2}\label{I2}
	\begin{cases}
	\partial_t\uv^{(k)} - \delta_v\Delta_{\Gamma}\uv^{(k)} + \beta L_v\uv^{(k)} = g(\uu^{(k-1)}, \uv^{(k-1)}),&x\in\Gamma, t>0,\\
	\uv^{(k)}(0,x) = v_0(x)\in L^{\infty}(\Gamma), &x\in\Gamma.
	\end{cases}
	\end{equation}
	Similarly, we construct a sequence of upper solutions $\{(\ou^{(k)}, \ov^{(k)})\}_{k\geq 0}:$
	\begin{equation}\tag{II.0}\label{II0}
	(\ou^{(0)}, \ov^{(0)}) = (\ou, \ov),
	\end{equation}
	and for all $k \geq 1$, $\ou^{(k)}$ and $\ov^{(k)}$ are the solutions of the following heat equation with inhomogeneous Robin boundary condition:
	\begin{equation}\label{II1}\tag{II.1}
	\begin{cases}
	\partial_t\ou^{(k)} - \delta_{u}\Delta\ou^{(k)} = 0,&x\in\Omega,t>0,\\
	\delta_u\frac{\partial\ou^{(k)}}{\partial\nu} + \alpha L_u\ou^{(k)}= f(\ou^{(k-1)}, \ov^{(k-1)}), &x\in\Gamma, t>0,\\
	\ou^{(k)}(0,x) = u_0(x),&x\in\Omega,
	\end{cases}
	\end{equation}
	and the following linear inhomogeneous equation:
	\begin{equation}\tag{II.2}\label{II2}
	\begin{cases}
	\partial_t\ov^{(k)} - \delta_v\Delta_{\Gamma}\ov^{(k)} + \beta L_v\ov^{(k)} = g(\ou^{(k-1)}, \ov^{(k-1)}), &x\in\Gamma, t>0,\\
	\ov^{(k)}(0,x) = v_0(x), &x\in\Gamma.
	\end{cases}
	\end{equation}
	
	The existence of unique sequences of lower and upper solutions $\uu^{(k)}$ and $\uv^{(k)}$ follows from classical arguments
	in an iterative way starting from \eqref{I0} and \eqref{II0}. Given for instance $(\uu^{(k-1)}, \uv^{(k-1)})$, system \eqref{I1} is a heat equation with inhomogeneous Robin boundary condition and bounded coefficients $L_u(t,x)\in L^{\infty}(\mathbb{R}_+\times\Gamma)$. Thus, the 
	existence of a unique weak solution in the sense of Definition \ref{def:weaksolution_1} follows from \cite{RMJ, RMJ1, NiPhD}, for instance. Moreover, the equation \eqref{I2} is a linear heat equation on a manifold without boundary and the existence of a unique weak solutions follows from \cite[Chapter 6]{Taylorbook}, for instance.
	
	Moreover, if $\uu^{(k-1)}$ and $\uv^{(k-1)}$ satisfy the regularity \eqref{regular_u} and \eqref{regular_v}, then by the locally Lipschitz properties of $f$ and $g$, we obtain $f(\uu^{(k-1)}, \uv^{(k-1)})$, $g(\uu^{(k-1)}, \uv^{(k-1)})\in L^{\infty}([0,T]\times \Gamma)$, which implies from (I.1) that $\uu^{(k)}$ satisfies \eqref{regular_u} and from (I.2) that $\uv^{(k)}$ satisfies \eqref{regular_v}.
	
An analogous argument can be applied to the equations \eqref{II1} and \eqref{II2} 	
in order to get the unique existence and the regular properties of $(\ou^{(k)}, \ov^{(k)})$. 
	
\begin{lemma}\label{lem:monotonicity}
The sequence $\{(\uu^{(k)}, \uv^{(k)})\}_{k\geq 0}$ is an monotone increasing sequence of lower solutions and $\{(\ou^{(k)}, \ov^{(k)})\}_{k\geq 0}$ is a monotone decreasing sequence of upper solutions. More precisely, for all $k\geq 0$,
		\begin{equation*}
		(\ou^{(k)}, \ov^{(k)}) \geq (\ou^{(k+1)}, \ov^{(k+1)}) \geq (\uu^{(k+1)}, \uv^{(k+1)}) \geq (\uu^{(k)}, \uv^{(k)})
		\end{equation*}
in the sense of Definition \ref{aecomparision}.
\end{lemma}
\begin{proof}\hfill\\
{The proof of this lemma follows \cite[Chapter 8]{Pao} 
with suitable changes to adapt to the present setting of our weak solutions.}
\end{proof}
	\medskip
	
	From Lemma \ref{lem:monotonicity} and with the help of the monotone convergence theorem, we have the following almost everywhere pointwise limits in $(0,T)\times \Omega$ and $(0,T)\times \Gamma$ respectively:
	\begin{equation}
	\lim\limits_{k\rightarrow \infty}(\uu^{(k)}, \uv^{(k)}) = (\uu^{*}, \uv^{*}) \text{ and } \lim\limits_{k\rightarrow \infty}(\ou^{(k)}, \ov^{(k)}) = (\ou^{*}, \ov^{*}).
	\label{a11}
	\end{equation}
	
	The following {\it a priori estimates} are uniform in $k$ pointwise for all times $t\in(0,T)$, and will allow us to pass to the limit $k\to\infty$:
	\begin{lemma}\label{lem:APrioriEstimate}
		The sequences $\{\ou^{(k)}\}_{k\geq 0}$ and $\{\ov^{(k)}\}_{k\geq 0}$ are uniformly bounded  in $k$ in $L^{\infty}(0,T;L^{\infty}(\Omega))\cap L^2(0,T;H^1(\Omega))$ and $L^{\infty}(0,T;L^{\infty}(\Gamma))\cap L^2(0,T;H^1(\Gamma))$, respectively, for any given $T>0$. Moreover, the sequence $\{(\ou^{(k)})^{\alpha}|_{\Gamma}\}_{k\geq 0}$ is bounded in $L^2(0,T;L^2(\Gamma))$. We also have analogous estimates for $\{\uu^{(k)}\}_{k\geq 0}$ and $\{\uv^{(k)}\}_{k\geq 0}$.
	\end{lemma}
	\begin{proof}
		We will prove only for $\{\ou^{(k)}\}_{k\geq 0}$ and $\{\ov^{(k)}\}_{k\geq 0}$. The estimate
		\begin{equation*}
		(\ou^{(k)}, \ov^{(k)}) \leq (\ou, \ov)
		\end{equation*}
		yields that $\{\ou^{(k)}\}_{k\geq 0}$ is bounded in $L^{\infty}(0,T;L^{\infty}(\Omega))$ and $\{\ov^{(k)}\}_{k\geq 0}$ is bounded in $L^{\infty}(0,T;L^{\infty}(\Gamma))$. More precisely, there exists $C_0 >0$ independent of $k$ such that
		\begin{equation*}
		\|\ou^{(k)}\|_{L^{\infty}(0,T;L^{\infty}(\Omega))} \leq C_0 \text{ and }\|\ov^{(k)}\|_{L^{\infty}(0,T;L^{\infty}(\Gamma))} \leq C_0 \text{ for all } k\geq 0, T>0.
		\end{equation*} 
		We now rewrite the equation for $\ou^{(k)}$ from \eqref{II1}
		\begin{equation*}
		\begin{cases}
		\partial_t\ou^{(k)} - \delta_u\Delta \ou^{(k)} = 0,\\
		\delta_u\frac{\partial\ou^{(k)}}{\partial\nu} + \alpha L_u \ou^{(k)} = -\alpha[k_u(\ou^{(k-1)})^{\alpha} - k_v(\ov^{(k-1)})^{\beta}] + \alpha L_u \ou^{(k-1)}.
		\end{cases}
		\end{equation*}
		By taking inner product with $\ou^{(k)}$ in $L^2(\Omega)$, we get
		\begin{align}\label{H1estimate}
		\frac{1}{2}\frac{d}{dt}\|\ou^{(k)}\|_{L^2(\Omega)}^2 + \delta_u\|\nabla\ou^{(k)}\|_{L^2(\Omega)}^2
		&= \int_{\Gamma}\left(- \alpha L_u \ou^{(k)} -\alpha[k_u(\ou^{(k-1)})^{\alpha} - k_v(\ov^{(k-1)})^{\beta}] + \alpha L_u \ou^{(k-1)}\right)\ou^{(k)}dS\nonumber\\
		&\leq \alpha\int_{\Gamma}k_v(\ov^{(k-1)})^{\beta}\ou^{(k)}dS + \alpha\int_{\Gamma}L_u\ou^{(k-1)}\ou^{(k)}dS
		\end{align}
		thanks to the nonnegativity of $\ou^{(k)}$, $\ou^{(k-1)}$, $k_u(t,x)\geq 0$ and $L_u = \alpha k_u \ou^{\alpha - 1} \geq 0$. In order to estimate the right hand side of \eqref{H1estimate}, we first have, { by using the modified Trace inequality $\|f\|_{\Gamma}^2 \leq \varepsilon\|\nabla f\|_{\Omega}^2 + C_{\varepsilon}\|f\|_{\Omega}^2$}
		\begin{equation}
		\begin{aligned}
		\alpha\int_{\Gamma}k_v(\ov^{(k-1)})^{\beta}\ou^{(k)}dS
		&\leq 2\alpha \|k_v\|_{\infty}\left(\int_{\Gamma}|\ov^{(k-1)}|^{2\beta}dS + \int_{\Gamma}|\ou^{(k)}|^2dS\right)\\
		&\leq 2\alpha \|k_v\|_{\infty}|\Gamma|\|\ov^{(k-1)}\|_{L^{\infty}(\Gamma)}^{2\beta} + \frac{\delta_u}{4}\|\nabla \ou^{(k)}\|_{L^2(\Omega)}^2 + C\|\ou^{(k)}\|_{L^\infty(\Omega)}^2.
		\end{aligned}
		\label{es1}
		\end{equation}
		Moreover, by using $L_u(t,x) = \alpha k_u \ou^{\alpha-1}(t,x) \leq \alpha \|k_u\|_{\infty}\|\ou\|_{L^{\infty}(0,T;L^{\infty}(\Omega))}^{\alpha - 1} =: C_1$, we get
		\begin{equation}
		\begin{aligned}
		\alpha\int_{\Gamma}L_u&\ou^{(k-1)}\ou^{(k)}dS\leq 2\alpha C_1\left(\int_{\Gamma}|\ou^{(k-1)}|^2dS + \int_{\Gamma}|\ou^{(k)}|^2dS\right)\\
		&\leq 2\alpha C_1\left(\frac{\delta_u}{8\alpha C_1}\|\nabla \ou^{(k-1)}\|_{L^2(\Omega)}^2 + C\|\ou^{(k-1)}\|_{L^2(\Omega)}^2 \right.
		\left.+ \frac{\delta_u}{8\alpha C_1}\|\nabla \ou^{(k)}\|_{L^2(\Omega)}^2 + C\|\ou^{(k)}\|_{L^2(\Omega)}^2\right)\\
		&\leq \frac{\delta_u}{4}\|\nabla \ou^{(k-1)}\|_{L^2(\Omega)}^2 + \frac{\delta_u}{4}\|\nabla \ou^{(k)}\|_{L^2(\Omega)}^2 + C(\|\ou^{(k-1)}\|_{L^{\infty}(\Omega)}^2 + \|\ou^{(k)}\|_{L^{\infty}(\Omega)}^2),
		\end{aligned}
		\label{es2}
		\end{equation}
		with a constant $C=C(C_1,\delta_u,\delta_v,\alpha)$.
		
		By applying \eqref{es1} and \eqref{es2} to \eqref{H1estimate}, we obtain
		\begin{multline}
		\frac{d}{dt}\|\ou^{(k)}\|_{L^2(\Omega)}^2 + \frac{\delta_u}{2}\|\nabla \ou^{(k)}\|_{L^2(\Omega)}^2\\
		\leq \frac{\delta_u}{4}\|\nabla \ou^{(k-1)}\|_{L^2(\Omega)}^2 + C\left(\|\ou^{(k)}\|_{L^{\infty}(\Omega)}^2 + \|\ou^{(k-1)}\|_{L^{\infty}(\Omega)}^2 + \|\ov^{(k-1)}\|_{L^{\infty}(\Gamma)}^{2\beta}\right).
		\label{es3}
		\end{multline}
		Integrating \eqref{es3} on $(0,T)$ yields
		\begin{align}
		\|\nabla \ou^{(k)}\|_{L^2(0,T;L^2(\Omega))}^2 &\leq \frac{2}{\delta_u}\|\ou^{(k)}(0)\|_{L^2(\Omega)}^2+\frac{1}{2}\|\nabla \ou^{(k-1)}\|_{L^2(0,T;L^2(\Omega))}^2\nonumber\\
		&\ \ + C\Bigl(\|\ou^{(k)}\|_{L^{\infty}(0,T;L^{\infty}(\Omega))}^2 + \|\ou^{(k-1)}\|_{L^{\infty}(0,T;L^{\infty}(\Omega))}^2 + \|\ov^{(k-1)}\|_{L^{\infty}(0,T;L^{\infty}(\Gamma))}^{2\beta}\Bigr)\nonumber\\
		&\leq \frac{1}{2}\|\nabla \ou^{(k-1)}\|_{L^2(0,T;L^2(\Omega))}^2 + \frac{2}{\delta_u}\|u_0\|_{L^2(\Omega)}^2+ C(2C_0^2 + C_0^{2\beta})\nonumber\\
		&\leq \frac{1}{2}\|\nabla \ou^{(k-1)}\|_{L^2(0,T;L^2(\Omega))}^2 + C.
		\label{es4}
		\end{align}
		Thus, we can have
		\begin{equation*}
		\begin{aligned}
		\|\nabla \ou^{(k)}\|_{L^2(0,T;L^2(\Omega))}^2 &\leq \frac{1}{2}\|\nabla \ou^{(k-1)}\|_{L^2(0,T;L^2(\Omega))}^2 + C\\
		&\leq \frac{1}{4}\|\nabla \ou^{(k-2)}\|_{L^2(0,T;L^2(\Omega))}^2 + C\left(1+\frac{1}{2}\right)\\
		&\leq \ldots\leq \frac{1}{2^k}\|\nabla \ou\|_{L^2(0,T;L^2(\Omega))}^2 + 2C\\
		\end{aligned}
		\end{equation*}
		Therefore, we have $\{|\nabla \ou^{(k)}|\}_{k\geq 0}$ is bounded in $L^2(0,T;L^2(\Omega))$ uniformly in $k$. By taking into account that $\{\ou^{(k)}\}_{k\geq 0}$ is bounded in $L^{\infty}(0,T;L^{\infty}(\Omega))\hookrightarrow L^2(0,T;L^2(\Omega))$, we see that $\{\ou^{(k)}\}_{k\geq 0}$ is uniformly bounded in $L^2(0,T;H^1(\Omega))$.
		
		We next prove that $\{(\ou^{(k)})^{\alpha}|_{\Gamma}\}_{k\geq 0}$ is bounded in $L^2(0,T;L^{2}(\Gamma))$. Indeed, adapting the estimate in Remark \ref{remark:Definition}, we get
		\begin{align*}
		\|(\ou^{(k)})^{\alpha}\|_{L^2(0,T;L^2(\Gamma))}^2 &= \iTG(\ou^{(k)})^{2\alpha}dSdt\\
		&\leq C\int_{0}^{T}\left(\|\ou^{(k)}\|_{L^{\infty}(\Omega)}^{2\alpha - 2}\|\nabla \ou^{(k)}\|_{L^2(\Omega)}^2 + \|\ou^{(k)}\|_{L^{\infty}(\Omega)}^{2\alpha}\right)dt\\
		&\leq C\|\ou^{(k)}\|_{L^{\infty}(0,T;L^{\infty}(\Omega))}^{2\alpha - 2}\|\nabla \ou^{(k)}\|_{L^{2}(0,T;L^2(\Omega))}^2 + C\|\ou^{(k)}\|_{L^{\infty}(0,T;L^{\infty}(\Omega))}^{2\alpha}.
		\end{align*}
		Thus, the boundedness of $\{(\ou^{(k)})^{\alpha}|_{\Gamma}\}_{k\geq 0}$ in $L^2(0,T;L^2(\Gamma))$ follows from the boundedness of $\{\ou^{(k)}\}_{k\geq 0}$ in $L^{\infty}(0,T;L^{\infty}(\Omega))$ and $L^2(0,T;H^1(\Omega))$.
		
		It remains to prove that $\{\ov^{(k)}\}_{k\geq 0}$ is bounded in $L^2(0,T;H^1(\Gamma))$. Multiplying the equation for $\ov^{(k)}$
		\begin{equation*}
		\partial_t\ov^{(k)} - \delta_v\Delta_{\Gamma} \ov^{(k)} + \beta L_v \ov^{(k)} = \beta[k_u(\ou^{(k-1)})^{\alpha} - k_v(\ov^{(k-1)})^{\beta}] + \beta L_v \ov^{(k-1)}
		\end{equation*}
		by $\ov^{(k)}$ in $L^2(\Gamma)$, we have
		\begin{multline}\label{es5}
		\frac{1}{2}\frac{d}{dt}\|\ov^{(k)}\|_{L^2(\Gamma)}^2 + \delta_v\|\nabla_{\Gamma}\ov^{(k)}\|_{L^2(\Gamma)}^2 + \beta\int_{\Gamma}L_v|\ov^{(k)}|^2dS\\
		= \beta\int_{\Gamma}[k_u(\ou^{(k-1)})^{\alpha} - k_v(\ov^{(k-1)})^{\beta}]\ov^{(k)}dS + \beta\int_{\Gamma}L_v\ov^{(k-1)}\ov^{(k)}dS\\
		\leq \beta \|k_u\|_{\infty}\int_{\Gamma}(\ou^{(k-1)})^{\alpha}\ov^{(k)}dS + \beta C_2\int_{\Gamma}\ov^{(k-1)}\ov^{(k)}dS
		\end{multline}
		since $k_v(\ov^{(k-1)})^{\beta}\ov^{(k)} \geq 0$ and $L_v = \beta k_v \ov^{\beta - 1} \leq \beta \|k_v\|_{\infty}\|\ov\|_{L^{\infty}(0,T;L^{\infty}(\Gamma))}^{\beta - 1} =: C_2$. By Young's inequality, we obtain
		\begin{equation*}
		\int_{\Gamma}(\ou^{(k-1)})^{\alpha}\ov^{(k)}dS \leq \frac{1}{2}\|(\ou^{(k-1)})^{\alpha}\|_{L^2(\Gamma)}^2 + \frac{1}{2}\|\ov^{(k)}\|_{L^2(\Gamma)}^2,
		\end{equation*}
		and
		\begin{equation*}
		\int_{\Gamma}\ov^{(k-1)}\ov^{(k)}dS \leq \frac{1}{2}\|\ov^{(k-1)}\|_{L^2(\Gamma)}^2 + \frac{1}{2}\|\ov^{(k)}\|_{L^2(\Gamma)}^2.
		\end{equation*}
		Therefore, it follows from \eqref{es5} that
		\begin{multline}
		\frac{d}{dt}\|\ov^{(k)}\|_{L^2(\Gamma)}^2 + 2\delta_v\|\nabla_{\Gamma}\ov^{(k)}\|_{L^2(\Gamma)}^2
		\leq \beta \|k_u\|_{\infty}\|(\ou^{(k-1)})^{\alpha}\|_{L^2(\Gamma)}^2 \\+ \beta( \|k_u\|_{\infty} + C_2)\|\ov^{(k)}\|_{L^2(\Gamma)}^2 + \beta C_2\|\ov^{(k-1)}\|_{L^2(\Gamma)}^2.
		\label{es6}
		\end{multline}
		By integrating \eqref{es6} over $(0,T)$, using that $\{(\ou^{(k)})^{\alpha}|_{\Gamma}\}_{k\geq 0}$ is uniformly bounded in $L^2(0,T;L^2(\Gamma))$ and $\{\ov^{(k)}\}_{k\geq 0}$ is uniformly bounded in $L^{\infty}(0,T;L^{\infty}(\Gamma))$, we conclude that $\{\ov^{(k)}\}_{k\geq 0}$ is uniformly bounded in $L^2(0,T;H^1(\Gamma))$. This completes the proof of the Lemma.
	\end{proof}
	
	\begin{proposition}\label{pro:Solutions}
		Both a.e. pointwise limits $(\uu^{*}, \uv^{*})$ and $ (\ou^{*}, \ov^{*})$ of \eqref{a11} are solutions of \eqref{e1}.
	\end{proposition}
	\begin{proof}
		{This proposition follows from the pointwise convergence, Lemma \ref{lem:APrioriEstimate} and the Dominated Convergence Theorem.}
	\end{proof}
	
	We are now ready to obtain the complete proof of Theorem \ref{theo:ExistenceAndUniqueness}.
	
	\begin{proof}[of Theorem \ref{theo:ExistenceAndUniqueness}]
		The existence of a solution is implied from Proposition \ref{pro:Solutions}. The non-negativity of solutions follows from the Comparison Theorem, see Lemma \ref{lem:comparison}, since $(\uu,\uv) = (0,0)$ is a lower solution. {To prove the uniqueness, we assume that $(u_1, v_1)$ and $(u_2, v_2)$ are two solutions with the same initial data. Thanks to Proposition \ref{UpperLower1}, $(A, B(t,x))$ is an upper solution, thus
				\begin{equation*}
					(u_1, v_1) \leq (A, B) \quad \text{ and } \quad (u_2, v_2) \leq (A, B).
				\end{equation*}
				Moreover, $B \in L^{\infty}([0,T]\times\Gamma)$ due to \eqref{defineB1}, \eqref{pi1} and \eqref{bound_k}. 
				We denote by $w = u_1 - u_2$ and $z = v_1 - v_2$ and have $w(0) = 0$ and $z(0) = 0$. Direct computations give
				\begin{equation*}
					\begin{aligned}
						\frac 12\frac{d}{dt}(\|w\|_{\Omega}^2 + \|z\|_{\Gamma}^2) + d_u\|w\|_{\Omega}^2 + d_v\|z\|_{\Gamma}^2
						&= -\alpha\int_{\Gamma}w(k_u(u_1^{\alpha} - u_2^{\alpha}) - k_v(v_1^{\beta} - v_2^{\beta}))dS \\
						&\quad+ \beta \int_{\Gamma}z(k_u(u_1^{\alpha} - u_2^{\alpha}) - k_v(v_1^{\beta} - v_2^{\beta}))dS.
					\end{aligned}
				\end{equation*}
By using $(u_1, v_1), (u_2, v_2)\leq (A, B)$, the mean value theorem for $u^{\alpha}$, $\alpha\ge1$ and $u^{\beta}$, $\beta\ge1$ and Cauchy's inequality, we obtain 
				\begin{equation*}
					\begin{aligned}
						\frac 12\frac{d}{dt}(\|w\|_{\Omega}^2 + \|z\|_{\Gamma}^2) + d_u\|w\|_{\Omega}^2 + d_v\|z\|_{\Gamma}^2
						&\leq \int_{\Gamma}(\alpha|w| + \beta|z|)(k_u|u_1^{\alpha} - u_2^{\alpha}| + k_v|v_1^{\beta} - v_2^{\beta}|)dS\\
						&\leq C(\alpha, \beta, \|k_u\|_{L^{\infty}}, \|k_v\|_{L^{\infty}}, A^{\alpha-1}, \|B\|_{L^{\infty}}^{\beta-1})(\|w\|_{\Gamma}^2 + \|z\|_{\Gamma}^2)\\
						&\leq C(\|w\|_{\Omega}^2 + \|z\|_{\Gamma}^2) + \frac{d_u}{2}\|\nabla w\|_{\Omega}^2
					\end{aligned}			
				\end{equation*}
				where the last inequality is obtained through a modified trace inequality $\|f\|_{\Gamma}^2 \leq \varepsilon\|\nabla f\|_{\Omega}^2 + C_{\varepsilon}\|f\|_{\Omega}^2$. As a consequence, by employing the classical Gronwall inequality we conclude that
				$w(t) = z(t) = 0$ for all $t\in [0,T]$ since $w(0) = z(0) = 0$, and hence  the proof of the uniqueness is completed.}
	\end{proof}
	
	\vskip 0.5cm
	\noindent{\bf Acknowledgements.} The authors would like to thank the referees for the useful comments which help to improve the presentation of the paper.
	
	The first author is supported by International Research Training Group IGDK 1754. This work has partially been supported by NAWI Graz.

\end{document}